\documentclass{amsart}

\usepackage{cite}

\usepackage{multirow}

\usepackage{amsmath, amssymb, amsthm, amscd, verbatim}

\usepackage{enumitem}

\newcommand{\na}{\mathrm{I}}
\newcommand{\nb}{\mathrm{I\!I}}
\newcommand{\nn}{\ast}

\newcommand{\R}  {{\mathbb R}}
\newcommand{\N}  {{\mathbb N}}
\newcommand{\Z}  {{\mathbb Z}}

\newcommand{\bx}{{\mathbf x}}
\newcommand{\by}{{\mathbf y}}
\newcommand{\bg}{{\boldsymbol{\gamma}}}
\newcommand{\be}{{\boldsymbol{\eta}}}

\newcommand{\X}{{\mathfrak X}}
\newcommand{\T}{{\mathfrak T}}

\newcommand{\Var}  {\operatorname{Var}} 

\newcommand{\scp}[2]{\langle #1, #2 \rangle}

\newcommand{\bsgamma}{\boldsymbol{\gamma}}

\newcommand{\bst}{\boldsymbol{t}}
\newcommand{\A}{{\mathcal{A}}}

\newcommand{\EE}{{\rm E}}

\DeclareMathOperator{\decay}{decay}
\DeclareMathOperator{\ran}{ran}

\DeclareMathOperator{\nes}{nest}
\DeclareMathOperator{\unr}{unr}

\newcommand{\std}  {\operatorname{std}} 

\theoremstyle{plain}
\newtheorem{lemma}{Lemma}[section]
\newtheorem{theo}{Theorem}[section]
\newtheorem{cor}{Corollary}[section]

\theoremstyle{definition}
\newtheorem{rem}{Remark}[section]
\newtheorem{exmp}{Example}[section]

\begin{document}

\title[Embeddings of Weighted Hilbert Spaces]
{Embeddings of Weighted Hilbert Spaces and Applications to 
Multivariate and Infinite-Dimensional Integration}

\author[Gnewuch]
{Michael Gnewuch}
\address{
Mathematisches Seminar\\
Christian-Albrechts-Universit\"at zu Kiel\\
Ludewig-Meyn-Str.\ 4\\
24098 Kiel\\ 
Germany}
\email{gnewuch@math.uni-kiel.de}

\author[Hefter]
{Mario Hefter }
\address{Fachbereich Mathematik\\
Technische Universit\"at Kaisers\-lautern\\
Postfach 3049\\
67653 Kaiserslautern\\
Germany}
\email{hefter@mathematik.uni-kl.de}

\author[Hinrichs]
{Aicke Hinrichs}
\address{
Institut f\"ur Analysis\\
Johannes-Kepler-Universit\"at Linz\\
Altenberger Str.\ 69\\
4040 Linz\\
Austria}
\email{aicke.hinrichs@jku.at}

\author[Ritter]
{Klaus Ritter}
\address{Fachbereich Mathematik\\
Technische Universit\"at Kaisers\-lautern\\
Postfach 3049\\
67653 Kaiserslautern\\
Germany}
\email{ritter@mathematik.uni-kl.de}

\date{June 28, 2016}

\keywords{High-dimensional integration,
infinite-dimensional integration,
embedding theorems,
Sobolev spaces,
reproducing kernel Hilbert spaces,
tractability}

\begin{abstract}
We study embeddings and norm estimates for
tensor products of weighted reproducing kernel Hilbert spaces.
These results lead to a transfer principle that is directly 
applicable to tractability studies of 
multivariate problems as integration and approximation, 
and to their infinite-dimensional counterparts.
In an application we consider
weighted tensor product Sobolev spaces of mixed smoothness 
of any integer order, equipped with the
classical, the anchored, or the ANOVA norm. 
Here we derive new results for multivariate and infinite-dimensional 
integration.
\end{abstract}

\maketitle

\section{Introduction}

The application of suitable embedding theorems in complexity studies of 
high-dimensional or infinite-dimensional
numerical problems has recently found an increased interest 
\cite{HR13,HEF14,Hefter20161,Hinrichs2015,2015arXiv151105674K,%
2015arXiv150602458K}.
The basic idea is simple: If two norms on a vector space are equivalent 
up to a constant $c \geq 1$, estimates for errors of algorithms 
measured with respect to one norm can increase only by 
this factor $c$ if measured with respect to the other norm.

To use this approach in tractability studies for high-dimensional problems,
a rather general setting is a scale $(H_s)_{s\in\N}$ of vector spaces 
of real-valued functions on the domain $D^s$, where $D$ is a given 
non-empty set. On each of the spaces $H_s$ we have two equivalent norms, 
so we get two scales of normed spaces. To transfer tractability results 
from one scale to the other, the following question becomes central: 
\emph{When are the two sequences of norms uniformly equivalent,}
i.e., when do we have 
\[
\sup_{s \in \N} \max \left\{\|\imath_s\|, \|\imath_s^{-1}\|\right\}
< \infty,
\] 
where $\imath_s$ and $\imath_s^{-1}$ denote the corresponding embeddings?

Infinite-dimensional integration deals with
the limiting case $s=\infty$. For each of the scales of normed
spaces we obtain a normed space of real-valued functions
with infinitely many variables, and the following
question becomes central:
\emph{When do the two spaces coincide and have equivalent
norms?}

In the case of tensor products of weighted reproducing kernel 
Hilbert spaces, this problem (with $s \in \N$ and $s=\infty$)
was first studied in \cite{HR13,HEF14}. 
In the present paper
we use a substantial extension of this abstract approach.
It allows us to deal, as particular instances,
with Hilbert spaces of functions of higher  
smoothness $r\ge1$, 
while only smoothness $r=1$ can be treated
within the framework provided 
in \cite{HR13, HEF14}.
The new framework presented in this paper is not only more general, but 
also more lucid than the one presented in \cite{HR13, HEF14}.
The whole approach is mainly motivated by the 
flexibility it provides in proving error bounds in the most convenient 
norm, while getting the result also for other interesting norms.

The main goal of this article is to solve open problems of 
tractability analysis and infinite-dimensional integration. In addition, 
it should serve to unify and simplify proofs of existing 
results.

Our prime example is the multivariate integration problem 
by means of deterministic or of randomized algorithms
in weighted 
Sobolev spaces of mixed smoothness $r$. Here it is natural and convenient
to treat three different norms: the standard norm, the ANOVA norm, and
the anchored norm, see, e.g., \cite[Sec.~A.2]{NW08}.
Roughly speaking, the anchored norm is known to be very well suited
for the analysis of deterministic algorithms, while the ANOVA
norm is much preferable for the analysis of randomized algorithms.
For instance, concerning the anchored and the standard norm,
Hickernell and Yue \cite[p.~2568]{YH06} 
state that the corresponding spaces ``have slightly
different norms \dots although the smoothness assumptions
are the same''.
In fact, all three norms are different, and in general we do not have
uniform equivalence of any two of theses norms, see \cite[Exmp.\ 4,
Thm.\ 1]{HR13}.
We will consider, however, a relaxed notion of equivalence in
\eqref{g9} that still allows to transfer tractability results.
These findings will then be extended to the case $s=\infty$, which 
allows to transfer results for infinite-dimensional integration.
The transference principle is formulated 
for $s\in \N\cup \{\infty\}$ in Theorem \ref{t3}.

With the help of the transference principle we obtain new 
results for multivariate and infinite-dimensional integration,
see Sections \ref{RMI} and \ref{RIDI}.
In particular, we summarize the known and new results for 
infinite-dimensional
integration by means of deterministic and randomized algorithms 
in Tables \ref{my-label1} and \ref{my-label2} in 
Section  \ref{RIDI}. These tables give a rather complete answer 
to the fundamental question whether randomization helps for 
infinite-dimensional integration in spaces that have been studied
recently by many authors, see Remark \ref{Rem_Ran_Det}.
A further new result deals with the
multivariate decomposition method (MDM), a general type of
algorithm for infinite-dimensional integration, which
was originally designed 
and analyzed
for anchored reproducing kernel Hilbert spaces, see \cite{KSWW10a, PW11}. 
For this type of spaces, it was known that the MDM achieves the 
optimal convergence rate, and the latter is determined explicitly
by the decay of the weights and by the corresponding convergence
rate for the one-dimensional problem ($s=1$).
According to Theorem \ref{Theo_UB_PW} the MDM can be 
used for general reproducing kernel Hilbert spaces: Under mild
assumptions we obtain the same result as for the particular case
of anchored spaces.
This general result would be hard to prove without our new embedding 
approach, cf., e.g., \cite{DG13}.

Although we apply our embedding 
results in this paper only to the multivariate and the 
infinite-dimensional integration problem, it is clear that these 
results can also be used in tractability studies of other multivariate 
or infinite-dimensional problems as, e.g., approximation of functions.

Let us sketch our approach and outline the structure 
of the paper. In general, we consider reproducing kernel Hilbert spaces 
with kernels of weighted tensor product form. The weights are given by a 
sequence $\bg=(\gamma_j)_{j \in \N}$ of positive real numbers that satisfy
\[
\sum_{j=1}^\infty \gamma_j < \infty,
\]
which is a common assumption in tractability analysis.

The univariate starting point are two pairs 
$(\|\cdot\|_{1,\na},\|\cdot\|_{2,\na})$ and
$(\|\cdot\|_{1,\nb},\|\cdot\|_{2,\nb})$ of seminorms on a vector
space $H$ of real-valued functions on $D$, both satisfying the same set 
of assumptions, see Section \ref{subsec:assumptions}. 
These assumptions ensure that for every $j \in
\N$ and $\nn \in \{\na,\nb\}$ there exists a reproducing kernel
$k_{\gamma_j,\nn}$ on $D \times D$ such that 
the norm $\|\cdot \|_{1+k_{\gamma_j,\nn}}$ on the Hilbert space 
$H( 1+k_{\gamma_j,\nn})$ with reproducing kernel $1+k_{\gamma_j,\nn}$ 
satisfies
\[
\|f\|^2_{1+k_{\gamma_j,\nn}} = 
\|f\|^2_{1,\nn} + \frac{1}{\gamma_j} \|f\|^2_{2,\nn}.
\]
Furthermore, $H=H(1+k_{\gamma_j,\nn})$ as vector spaces,
so that we have equivalence of all the 
norms $\|\cdot\|_{1+k_{\gamma_j,\nn}}$.
To provide some intuition on the role of the weights, observe that
$\lim_{j \to \infty} 
\|f\|^2_{1+k_{\gamma_j,\nn}} = \infty$ unless $\|f\|_{2,\nn} =0$.
The latter turns out to be equivalent to $f$ being constant.
The basic embedding result and norm estimate for
functions of a single variable is derived in Section \ref{ss1}.

In Section \ref{s2.3} we then consider spaces of functions of 
finitely many variables. The reproducing kernels $K^{\bg,\nn}_s$ on 
$D^s \times D^s$ are the tensor products of the one-dimensional
kernels  $1+k_{\gamma_j,\nn}$ for $j=1,\dots,s$.
It follows that $H(K^{\bg,\nn}_s)$, as a vector space, does neither
depend on $\nn$ nor on $\gamma_1, \dots, \gamma_s$. This yields
the equivalence of all the norms $\|\cdot\|_{K^{\bg,\nn}_s}$ with
$\nn \in \{\na,\nb\}$ and $\bg$ as previously
on $H_s=H(K^{\bg,\nn}_s)$, which is a space
of real-valued functions on $D^s$.

In general, summability of the weights does not imply uniform equivalence 
of the norms on $H(K^{\bg,\na}_s)$ and $H(K^{\bg,\nb}_s)$, see
\cite{HR13}. As a remedy, 
we consider $c\bg = (c \gamma_j)_{j \in \N}$ with $c>0$, and we observe 
that the corresponding norms are monotonically decreasing functions of $c$.
Using $\imath^{{\be},\bg}_s$ to denote the embedding of 
$H(K_s^{{\be},\nb})$ into $H(K_s^{\bg,\na})$, we prove that there exists a
constant $0<c_0<1$, which only depends on the two pairs of seminorms, 
such that
\begin{equation}\label{g9}
\sup_{s \in \N} \max \left\{ \|\imath^{c_0\bg,\bg}_s\|, \|
(\imath^{c_0^{-1} \bg,\bg}_s)^{-1}\|,
\|\imath^{\bg,c_0^{-1}\bg}_s\|, \|
(\imath^{\bg,c_0\bg}_s)^{-1}\| \right\} < \infty
\end{equation}
for all sequences of summable weights, see Corollary \ref{c1}.

In Section \ref{s2.4} we proceed to spaces of functions of infinitely
many variables. The limit $K^{\bg,\nn}_\infty = 
\lim_{s \to \infty} K^{\bg,\nn}_s$
defines a reproducing kernel for a space of functions
of infinitely many variables. The domain is the sequence space $D^\N$, 
or, for technical reasons, a proper subset thereof, 
and it does not depend on $\nn$. 

In general, we do not have
$H(K_\infty^{{\bg},\na}) = H(K_\infty^{\bg,\nb})$, see
\cite{HR13}, but a similar approach as for $s \in \N$ is possible.
In fact, let $\imath^{{\be},\bg}_\infty$ denote the embedding of 
$H(K_\infty^{{\be},\nb})$ into $H(K_\infty^{\bg,\na})$, provided
that the corresponding domains do coincide and
$H(K_\infty^{{\be},\nb}) \subseteq H(K_\infty^{\bg,\na})$.
These domains turn out to be invariant with respect to 
multiplication of the weights by any constant,
and with $0<c_0<1$ as previously we obtain 
\[
\max \left\{ \|\imath^{c_0\bg,\bg}_\infty\|, \|
(\imath^{c_0^{-1} \bg,\bg}_\infty)^{-1}\|,
\|\imath^{\bg,c_0^{-1}\bg}_\infty\|, \|
(\imath^{\bg,c_0\bg}_\infty)^{-1}\| \right\} < \infty
\]
for all sequences of summable weights, see Corollary \ref{c2}.

Fortunately,
most of the known results on tractability 
and on infinite-dimensional integration
are invariant with respect to 
a multiplication of the weights with a constant,
and given this invariance the transfer of results
between the two scales of spaces does not require any further effort.

The basic difference between the approach in the present paper together 
with \cite{HR13,HEF14} and the approach in 
\cite{Hefter20161,Hinrichs2015,2015arXiv151105674K,
2015arXiv150602458K} to the analysis of embeddings and
equivalences of norms is as follows. The former papers
consider an abstract setting, which deals with reproducing
kernel Hilbert spaces in tensor product form and, accordingly,
with product weights. The latter approach deals
with specific spaces, namely, weighted Sobolev
spaces of mixed smoothness of order $r=1$, and with specific norms. 
However, the latter approach is not restricted to product weights,
and it allows to measure derivatives in any $L_p$-norm with 
$1 \leq p \leq \infty$.
The extremal cases $p \in \{1,\infty\}$ are analyzed
first, and then the results are extended to arbitrary $p$ by means
of interpolation theory, see \cite{Hinrichs2015}.

In Section \ref{s3} we start to apply the embedding results and
the norm estimates from Section \ref{SEC_EMB}. We formally
introduce the integration problem 
for $s\in \N$ and for $s = \infty$
and we discuss the notions of randomized and deterministic 
algorithms and the corresponding minimal errors. 
In Theorem  \ref{t3} the general error estimates for 
our transfer principle are stated. Finally, in Sections \ref{RMI} 
and \ref{RIDI}, specific known and new results for finite- and 
infinite-dimensional integration in both the deterministic and the
randomized setting are presented together with a discussion.

\section{Embedding Results and Norm Estimates}\label{SEC_EMB}

We present some abstract assumptions for Hilbert spaces of 
functions of a single variable, and embedding results and norm
estimates are first of all
derived in this setting. The tensor product structure of the 
function spaces allows to extend the results to the multivariate
case and to spaces of function with infinitely many variables.

\subsection{Assumptions}\label{subsec:assumptions}

We frequently use basic results from \cite{Aro50} about reproducing
kernels $K$ and the corresponding Hilbert spaces $H(K)$
without giving further reference.
We denote the space of constant functions (on a given domain) by $H(1)$; 
here $1$ denotes the constant kernel that only takes the function value 
one. Throughout the paper we do
not distinguish between a function in $H(1)$ and 
its constant function value.
Henceforth we assume that
\begin{enumerate}[label=(A\arabic*)]
\item\label{a1}
$H$ is a vector space of real-valued functions on a set $D \neq
\emptyset$ with $H(1) \subsetneq H$
\end{enumerate}
and
\begin{enumerate}[label=(A\arabic*)]
\setcounter{enumi}{1}
\item\label{a2}
$\|\cdot\|_1$ and $\|\cdot\|_2$ are seminorms on $H$,
induced by symmetric bilinear forms $\langle\cdot,\cdot\rangle_1$ and 
$\langle\cdot,\cdot\rangle_2$, such that
$\|1\|_1=1$ and $\|1\|_2=0$.
\end{enumerate}
Let 
\begin{equation}\label{eq3}
\|f\|_H = \left(\|f\|_1^2 + \|f\|_2^2\right)^{1/2}
\end{equation}
for $f \in H$. Henceforth we also assume that
\begin{enumerate}[label=(A\arabic*)]
\setcounter{enumi}{2}
\item\label{a3}
$\|\cdot\|_H$ is a norm on $H$ that turns this space into a
reproducing kernel Hilbert space, and there exists a constant $c
\geq 1$ such that
\begin{equation}\label{eq2}
\|f\|_H \leq c\left(\left|\langle f,1\rangle_1\right|+\|f\|_2\right)
\end{equation}
for all $f \in H$.
\end{enumerate}

Condition \eqref{eq2} is equivalent to the 
fact that $\|\cdot\|_H$ and
$\left|\langle\cdot,1\rangle_1\right|+\|\cdot\|_2$ 
are equivalent norms on $H$. 
Note also that 
$\|1\|_2=0$
is equivalent to $\|f+c\|_2=\|f\|_2$ for all $c\in \R$ and 
$f\in H$.
Actually, \eqref{eq2} implies that $\|f\|_2=0$ if and only 
if $f$ is constant.

\begin{rem}\label{r0}
A  setting  that is frequently studied in the literature on
tractability and infinite-dimensional integration
is that of a reproducing kernel Hilbert space  $H(1+k)$, where
\begin{enumerate}[label=(B\arabic*)]
\item\label{b1}
$k\neq 0$ is a reproducing kernel on $D\times D$ for some set 
$D\neq\emptyset$ such that $H(1)\cap H(k)=\{0\}$.
\end{enumerate}
A canonical pair of seminorms on the vector space $H=H(1+k)$ is derived by 
\[
\|f\|_{1} = |P(f)|
\]
and 
\[
\|f\|_{2} = \|f-P(f)\|_k,
\]
where $P$ denotes the orthogonal projection of $H(1+k)$ onto
$H(1)$. Observe that $\|\cdot\|_H = \|\cdot \|_{1+k}$.
Obviously, we have \ref{a1}, \ref{a2}, and \ref{a3}.
\end{rem}

\subsection{Functions of a Single Variable}\label{ss1}

We use the seminorms $\|\cdot\|_1$ and $\|\cdot\|_2$
to construct a family of reproducing kernels on $D\times D$.

\begin{lemma}\label{lem10}
For every $\gamma>0$ there exists a uniquely determined reproducing 
kernel $k_\gamma$ on $D\times D$ such that
\[
H(1+k_\gamma)=H
\]
and
\begin{align}\label{eqlem10}
\|f\|_{1+k_\gamma}^2
=\|f\|_1^2+\frac{1}{\gamma}\|f\|_2^2
\end{align}
for all $f\in H$. Moreover, the norms $\|\cdot\|_H$ and 
$\|\cdot\|_{1+k_\gamma}$ are equivalent,
and 
\[
H(1)\cap H(k_{\gamma})=\{0\}.
\]
\end{lemma}

\begin{proof}
Fix $\gamma > 0$, and put
\[
\|f\| = 
\Bigl( \|f\|_1^2+\frac{1}{\gamma}\|f\|_2^2 \Bigr)^{1/2} 
\]
for $f \in H$.
Observe that $\|\cdot\|$ 
is a norm on $H$, which is induced by a symmetric bilinear form and
is equivalent to $\|\cdot\|_H$.

Let $H$ be equipped with the norm $\|\cdot\|$.
Since the norms $\|\cdot\|_H$ and $\|\cdot\|$ are equivalent
this is again a reproducing kernel Hilbert space.
For the orthogonal complement of $H(1)$ in this space we have
\[
H(1)^\perp = \{ f \in H \mid \scp{f}{1}_1 = 0\},
\]
and $\scp{f}{1}_1$ is the orthogonal projection of $f \in H$
onto $H(1)$. 
Furthermore, if $k_\gamma$ denotes the reproducing kernel of
$H(1)^\perp$, considered as a subspace of $H$,
then $1+k_\gamma$ is the reproducing kernel of $H$.
\end{proof}

In the following, $k_\gamma$ always denotes the reproducing kernel 
from Lemma~\ref{lem10}.
Typically, we do not refer to the explicit form of $k_\gamma$.

In 
Remark \ref{rem1}, see also Remark \ref{r3a}, we present an important
case, where 
there exists a reproducing kernel $k$ on $D \times D$ such that
\begin{equation}\label{g1}
k_\gamma=\gamma\cdot k 
\end{equation}
for all $\gamma>0$.
See, however, Remark~\ref{r1a} for another important
case, where we do not have this property -- not 
even for only two different values of $\gamma$.

\begin{rem}\label{rem1}
Suppose that the seminorm $\|\cdot\|_1$ is given in terms of a bounded 
linear functional $\xi$ on $H$, i.e.,
\begin{align*}
\|f\|_1
=\left|\xi(f)\right|
\end{align*}
for all $f\in H$. 
Let $k$ be the reproducing kernel on $D\times D$ such that
\begin{align*}
H(k)=\{f\in H\mid \xi(f)=0\}
\end{align*}
and
\begin{align*}
\|f\|_k
=\|f\|_{1+k_1}
= \|f\|_H
\end{align*}
for all $f\in H(k)$.
Clearly $H(1+k) =H$ and
$H(1)\cap H(k) =\{0\}$.
Furthermore, $\|f\|_2 = \|f-\xi(f)\|_k$ for all $f\in H$.
Consequently,
\begin{align*}
\|f\|_{1+\gamma\cdot k}^2
=\left|\xi(f)\right|^2+\frac{1}{\gamma}\|f-\xi(f)\|_{1+k_1}^2
=\|f\|^2_1+\frac{1}{\gamma}\|f\|_2^2
\end{align*}
for all $f\in H$ and $\gamma>0$,
which implies \eqref{g1}.
Moreover, by definition of $k$,
\[
\xi(k(\cdot,x)) = 0
\]
for every $x \in D$.

We add that the kernel $k$ satisfies \ref{b1} and the projection 
$P$ in Remark \ref{r0} is equal to the linear functional $\xi$.  
\end{rem}

\begin{rem}\label{r1a}
We show that the case considered in Remark \ref{rem1} is the only case
where $k_\gamma=\gamma\cdot k$ for at least two different values $\gamma$. 
Suppose that $\|\cdot\|_1$ is a seminorm induced by a symmetric
bilinear form that is not induced by a functional as in Remark \ref{rem1}.
Since the Cauchy-Schwarz inequality holds for such seminorms it follows
that the null space $ N=\{ f \in H \mid \|f\|_1=0 \}$ is a linear space
of codimension at least 2. Therefore there exists $f\in H$ with
$\langle f,1 \rangle_1 = 0$ and $\|f\|_1 \neq 0$.
Let $0 < \gamma_1 < \gamma_2$. We show that there 
exists no reproducing kernel $k$ on $D \times D$ such that 
$k_{\gamma_i} = \gamma_i \cdot k$ for $i=1,2$. Assuming the contrary
we obtain for $i \in \{1,2\}$ that 
\[
\|f\|_1^2 + \frac{1}{\gamma_i} \|f\|_2^2 =
\|f\|_{1+ k_{\gamma_i}}^2 =
\langle f,1 \rangle_1^2 + \| f - \langle f,1 \rangle_1\|_{k_{\gamma_i}}^2
= \frac{1}{\gamma_i} \|f\|_k^2.
\] 
Since $\|f\|_1 \neq 0$, this is a contradiction.
\end{rem}

Let $\|\cdot\|_{1,\na}$ and $\|\cdot\|_{2,\na}$ as well as 
$\|\cdot\|_{1,\nb}$ and $\|\cdot\|_{2,\nb}$ be two pairs 
of seminorms on $H$, both satisfying 
\ref{a2} and \ref{a3}. 
In the following we will compare the resulting norms according to 
\eqref{eqlem10}, and
we denote the corresponding reproducing kernels 
according to Lemma~\ref{lem10} by $k_{\gamma,\na}$ and 
$k_{\gamma,\nb}$, respectively.

\begin{theo}\label{lem1}
There exists a constant $0 < c_0 < 1$ with the following property
for every $\gamma > 0$.
For all $f\in H$,
\begin{align*}
\|f\|_{1+k_{\gamma,\na}}
\leq (1+\gamma)^{1/2} \cdot \|f\|_{1+k_{c_0 \gamma,\nb}}.
\end{align*}
\end{theo}

\begin{proof}
Let $\|\cdot\|_{H,\na}$ and $\|\cdot\|_{H,\nb}$ denote
the norms on $H$ that are derived from the corresponding pairs
of seminorms via \eqref{eq3}. The two norms are equivalent,
which follows from the closed graph theorem and the assumption
that both of the norms turn $H$ into a reproducing kernel Hilbert
space.
Without loss of generality we assume the estimate
\eqref{eq2} to hold for both pairs of seminorms with a common 
constant $c\geq 1$,
and also 
\[
\|f\|_{H,\na} \leq c \|f\|_{H,\nb}
\]
to hold for all $f \in H$.

Let $f\in H$.
Then
\begin{align*}
\|f\|_{1,\na}
&\leq \|f-\langle f, 1\rangle_{1,\nb}\|_{1,\na}+
\|\langle f,1\rangle_{1,\nb}\|_{1,\na}\\
&\leq c\|f-\langle f, 1\rangle_{1,\nb}\|_{H,\nb} +
\left|\langle f,1\rangle_{1,\nb}\right|
\leq \|f\|_{1,\nb}+c^2\|f\|_{2,\nb}
\end{align*}
and
\begin{align*}
\|f\|_{2,\na} 
&= \|f-\langle f, 1\rangle_{1,\nb}\|_{2,\na} 
\leq c \|f-\langle f, 1\rangle_{1,\nb}\|_{H,\nb} 
\\
& \leq c^2 
\left( | \langle f - \langle f,1 \rangle_{1,\nb}, 1
\rangle_{1,\nb}| +
\|f - \langle f,1 \rangle_{1,\nb}\|_{2,\nb}
\right)
\leq c^2 \|f\|_{2,\nb}. 
\end{align*}
It follows that
\begin{align*}
\|f\|_{1,\na}^2+\frac{1}{\gamma}\|f\|_{2,\na}^2
&\leq \left(\|f\|_{1,\nb}+c^2\|f\|_{2,\nb}\right)^2+
\frac{c^4}{\gamma}\|f\|_{2,\nb}^2\\
&\leq (1+\gamma)\|f\|_{1,\nb}^2+\left(1+\frac{1}{\gamma}\right)
c^4\|f\|_{2,\nb}^2+\frac{c^4}{\gamma}\|f\|_{2,\nb}^2\\
&=(1+\gamma)\biggl(\|f\|_{1,\nb}^2+\frac{\left(1+\frac{1}{\gamma}
\right)c^4+\frac{c^4}{\gamma}}{1+\gamma}\|f\|_{2,\nb}^2\biggr)\\
&\leq(1+\gamma)\Bigl(\|f\|_{1,\nb}^2+
\frac{2c^4}{\gamma}\|f\|_{2,\nb}^2\Bigr).
\end{align*}
The constant $c_0 = 1/(2c^4)$ therefore has the property as claimed.
\end{proof}

As we impose the same set of assumptions on both pairs of seminorms, 
results like Theorem~\ref{lem1} are also valid in reverse order.

\begin{exmp}\label{exa1}
Fix $r \in \N$ and consider the Sobolev space
\begin{align*}
W^{r,2}[0,1]
=\{f\in L^2[0,1] \mid f^{(\nu)}\in L^2[0,1], 1\leq\nu\leq r\},
\end{align*}
where $f^{(\nu)}$ denotes the $\nu$th distributional derivative of
$f$. Moreover, consider three pairs of seminorms on this space,
given by
\begin{align*}
\|f\|_{1,\text{S}}^2
&=\int^1_0 \left| f(y) \right|^2  \,{\rm d}y,\\
\|f\|^2_{2,\text{S}}
&=\sum_{\nu=1}^r \int^1_0 |f^{(\nu)}(y)|^2 \,{\rm d}y\\
\intertext{and}
\|f\|_{1,\pitchfork}
&=|f(a)|,\\
\|f\|^2_{2,\pitchfork}
&=\sum_{\nu=1}^{r-1} |f^{(\nu)}(a)|^2 + \int^1_0 |f^{(r)}(y)|^2 
\,{\rm d} y \\
\intertext{with $a \in [0,1]$, as well as}
\|f\|_{1,\text{A}}
&=\left| \int^1_0 f(y) \,{\rm d}y \right|,\\
\|f\|^2_{2,\text{A}}
&=\sum_{\nu=1}^{r-1} \left|\int_0^1 f^{(\nu)}(y) \, {\rm d} y 
\right|^2 + \int^1_0 |f^{(r)}(y)|^2 \,{\rm d}y
\end{align*}
for $f\in W^{r,2}[0,1]$. 
See, e.g., \cite[Sec.~A.2]{NW08}.
The assumptions \ref{a1} and \ref{a2} are obviously
satisfied for $H=W^{r,2}[0,1]$ and each of these pairs of
seminorms.

Let
\[
\|f\|_{H,\nn} = 
\bigl(\|f\|_{1,\nn}^2 + \|f\|_{2,\nn}^2\bigr)^{1/2}
\]
for $\nn\in\{\text{S},\pitchfork,\text{A}\}$. 
It is well known that we get three equivalent norms in this way, 
each of which turns $H$ into a reproducing kernel Hilbert
space.
To establish \ref{a3} it remains to verify \eqref{eq2}.
The latter trivially holds with $c=1$ for
$\nn\in\{\pitchfork,\text{A}\}$,
since we have $|\langle f, 1 \rangle_{1,\nn}| = \|f\|_{1,\nn}$ in
these two cases. Finally, we get \eqref{eq2} in the case $\nn =
\text{S}$ from the trivial estimates
$\|f\|_{1,\text{S}} \leq \|f\|_{H,\text{S}}$
and
$\|f\|_{2,\text{A}} \leq \|f\|_{2,\text{S}}$
together with the equivalence of $\|\cdot\|_{H,\text{A}}$ and
$\|\cdot\|_{H,\text{S}}$.
In all three cases we denote the reproducing kernels according 
to Lemma~\ref{lem10} by $k_{\gamma,\nn}$. 

For $\nn \in \{\pitchfork,\text{A}\}$
we are in the situation of Remark~\ref{rem1}, so that
$k_{\gamma,\nn} = \gamma \cdot k_{1,\nn}$.
The norm on $H(1+k_{\gamma,\pitchfork})$ is given by
\[
\|f\|_{1+k_{\gamma,\pitchfork}}^2 =
|f(a)|^2 + \frac{1}{\gamma} \left(
\sum_{\nu=1}^{r-1} |f^{(\nu)}(a)|^2 + \int^1_0 |f^{(r)}(y)|^2 
\,{\rm d} y \right),
\] 
which corresponds to the anchored ($\pitchfork$)
decomposition of $f$, see, e.g., \cite{KSWW10b}.
The norm on $H(1+k_{\gamma,\text{A}})$ is given by
\[
\|f\|_{1+k_{\gamma,\text{A}}}^2 =
\left| \int^1_0 f(y) \,{\rm d}y \right|^2
+ \frac{1}{\gamma} \left(
\sum_{\nu=1}^{r-1} \left|\int_0^1 f^{(\nu)}(y) \, {\rm d} y 
\right|^2 + \int^1_0 |f^{(r)}(y)|^2 \,{\rm d}y
\right),
\] 
which corresponds to the ANOVA ($\text{A}$)
decomposition of $f$, see, e.g., \cite{DG13,KSWW10b}.

For $\nn = \text{S}$ the norm on $H(1+k_{\gamma,\text{S}})$ is
given by
\[
\|f\|_{1+k_{\gamma,\text{S}}}^2 =
\int^1_0 \left| f(y) \right|^2  \,{\rm d}y
+ \frac{1}{\gamma} \left(
\sum_{\nu=1}^r \int^1_0 |f^{(\nu)}(y)|^2 \,{\rm d}y
\right).
\]
In particular, for $\gamma=1$ we obtain
a standard ($\text{S}$) norm on the Sobolev space. 
The seminorm $\|\cdot\|_{1,\text{S}}$ is not induced by a bounded
linear functional. Hence we are not in the situation of Remark
\ref{rem1}, and
according to Remark \ref{r1a} there exists no reproducing
kernel $k$ such that $k_{\gamma,\text{S}}=\gamma\cdot k$ even for
only two different values of $\gamma$.
\end{exmp}

\begin{exmp}\label{exa2}
We discuss two natural modifications of the setting in
Example \ref{exa1} with $H=W^{r,2}[0,1]$ in the case $r \geq 2$.
At first, let $\|\cdot\|_{1,\nn}$ be given as previously,
but 
\begin{align*}
\|f\|^2_{2,\text{S}'}
=\|f\|^2_{2,\pitchfork'}
=\|f\|^2_{2,\text{A}'}
=\int^1_0 |f^{(r)}(x)|^2 \,{\rm d} x.
\end{align*}
For $\nn\in\{\text{S},\pitchfork,\text{A}\}$ the assumption \ref{a3}
is not satisfied for the seminorms 
$\|\cdot\|_{1,\nn}$ and $\|\cdot\|_{2,\nn'}$. 
In fact, for $\nn = \text{S}$ we do not have
\eqref{eq2}, and $\|\cdot\|_{1,\nn} + \|\cdot\|_{2,\nn'}$ does not
even define a norm on $H$ for $\nn\in\{\pitchfork,\text{A}\}$.

Now, we consider the seminorms 
$\|\cdot\|_{1,\nn'}$ and $\|\cdot\|_{2,\nn'}$, where
\begin{align*}
\|f\|_{1,\text{S}'}^2
&=\sum_{\nu=0}^{r-1} \int^1_0 |f^{(\nu)}(y)|^2 \,{\rm d}y,\\
\|f\|_{1,\pitchfork'}^2
&=\sum_{\nu=0}^{r-1} |f^{(\nu)}(a)|^2,\\
\|f\|_{1,\text{A}'}^2
&=\sum_{\nu=0}^{r-1} \left|\int_0^1 f^{(\nu)}(y) \, {\rm d} y 
\right|^2 .
\end{align*}
Clearly,
\[
\|f\|_{1,\nn'}^2 + \|f\|_{2,\nn'}^2 = 
\|f\|_{1,\nn}^2 + \|f\|_{2,\nn}^2 
\]
for $f \in H$, but we do not have \eqref{eq2}
for any $\nn\in\{\text{S},\pitchfork,\text{A}\}$.
Still, Lemma \ref{lem10} is valid.
As we will see in Remark \ref{Remark5}, Theorem \ref{lem1} is not valid 
for the pairs of seminorms
$(\| \cdot \|_{\nn', 1}, \| \cdot 
\|_{\nn', 2})$ and $(\| \cdot \|_{\nn, 1}, \| \cdot \|_{\nn, 2})$
for any $\nn \in \{\text{S},\pitchfork,\text{A}\}$ and also not 
for $(\| \cdot \|_{\text{S}, 1}, \| \cdot \|_{\text{S}', 2})$ and 
$(\| \cdot \|_{\text{S}, 1}, \| \cdot \|_{\text{S}, 2})$.
\end{exmp}

\begin{rem}\label{r3a}
In contrast to the setting studied so far, we now take a reproducing 
kernel Hilbert space $H$ with $1 \in H$ as the overall starting point. 
We use $\|\cdot\|$ to denote the norm on $H$.
Consider a bounded linear functional $\xi$ on $H$ such that
$\xi(1)=1$. Define two seminorms on $H$ by
\[
\|f\|_1
=\left|\xi(f)\right|
\]
and
\[
\|f\|_2
=\|f-\xi(f)\|
\]
for $f\in H$.
Obviously, the assumption \ref{a1} and \ref{a2} are
satisfied, and we have equivalence of the norms $\|\cdot\|$ and
$\|\cdot\|_H$, given by \eqref{eq3}. It follows that $H$, equipped
with $\|\cdot\|_H$, is a reproducing kernel Hilbert space, too.
Furthermore, 
\[
\|f\|_H^2= \langle f,1\rangle_1^2+\|f\|_2^2, 
\]
so that \eqref{eq2} is satisfied with $c=1$.
Altogether, this yields \ref{a3}, and we are in the situation of
Remark \ref{rem1}.
\end{rem}

\begin{rem}\label{r3}
The analysis in \cite{HR13} is based on the assumptions 
\ref{b1} in Remark~\ref{r0} and
\begin{enumerate}[label=(B\arabic*)]
\setcounter{enumi}{1}
\item\label{b2}
$\|\cdot\|$ is a seminorm on $H(1+k)$, induced by a
symmetric bilinear form $\scp{\cdot}{\cdot}$
that satisfies
$\|1\|=1$ 
as well as
$\|f\| \leq c \|f\|_k$ for every $f \in H(k)$ with some
constant $c>0$.
\end{enumerate}
We denote the seminorms from Remark \ref{r0} by 
$\|\cdot\|_{1,\na}$ and  $\|\cdot\|_{2,\na}$.
A second pair of seminorms on $H$ is defined by
\[
\|f\|_{1,\nb} = \|f\|
\]
and
\[
\|f\|_{2,\nb} = \|f\|_{2,\na}. 
\]
The latter pair of seminorms also satisfies 
\ref{a2},
and $\|\cdot\|_{H,\nb}$ is shown to be equivalent to
$\|\cdot\|_{H,\na}$ in \cite[Lem.~1]{HR13}. To establish \ref{a3}
     for $\|\cdot\|_{H,\nb}$
it therefore remains to verify \eqref{eq2}.
To this end we consider for the moment the seminorm 
$\left|\langle \cdot,1\rangle\right|$, which also satisfies \ref{b2}, 
instead of $\|\cdot\|$. The previously mentioned equivalence of the 
corresponding norms then yields \eqref{eq2}.

We conclude that the setting from \cite{HR13}, namely
\ref{b1} and \ref{b2} and the seminorms
$\|\cdot\|_{1,\nn}$ and $\|\cdot\|_{2,\nn}$ for $\nn \in \{\na,\nb\}$,
implies \ref{a1}--\ref{a3} for both pairs of these seminorms. 
Actually the setting from \cite{HR13} is stronger then the present
setting, since $\|f\|_{2,\nb}$ and $\|f\|_{2,\na}$ are assumed to
coincide in \cite{HR13}.
Observe that the latter does not hold  
in the situation of Example~\ref{exa1}, if $r\geq 2$. 
\end{rem}

\begin{exmp}\label{e33}
Spaces of smooth periodic functions are often defined in terms of the 
decrease of the Fourier coefficients. Assume that $D=
[0,1]$ and that 
a sequence $(\omega_h)_{h\in\Z}$ of positive numbers is given with 
$\omega_0=1$ and $\omega_h\to\infty$ as $|h|\to\infty$. Let $H$ be the 
Hilbert space of all 
$f\in L^2[0,1]$ with finite norm
\[ 
 \|f\|_H^2 = \sum_{h\in\Z} | \hat{f}(h)|^2 \, \omega_h,
\]
where
\begin{align*}
\hat{f}(h)
=\int_0^{1} f(t)e^{- 2\pi i h t}\,\mathrm dt. 
\end{align*}
We consider the pair of seminorms on $H$ given by
\[
\|f\|_{1} = | \hat{f}(0)|
\]
and
\[
\|f\|_{2}^2 = \sum_{h\neq 0} | \hat{f}(h)|^2 \, \omega_h.
\]
See, e.g., \cite[Sec.~A.1]{NW08}.
To obtain Hilbert spaces of $1$-periodic
real-valued functions and continuous 
function evaluations we need to assume that
\begin{align}\label{persum}
\sum_{h\in\Z}\frac{1}{\omega_h}<\infty.
\end{align}
This follows from the fact that the Cauchy-Schwarz inequality implies 
that in this case the Fourier series of $f\in H$ converges absolutely 
and uniformly to $f$. In this case the assumptions \ref{a1}, \ref{a2}, 
and \ref{a3} are easily verified. 
Specific examples are the periodic Sobolev spaces 
$H^r[0,1]$ for $r\in (0,\infty)$ that 
correspond to the choice $\omega_h=
\max\{1,|h|^{2r}\}$.
Here condition \eqref{persum} is equivalent to $r>1/2$.

We are in the situation of Remark~\ref{rem1}, so 
that $k_{\gamma} = \gamma \cdot k_{1}$.
The norm on $H(1+k_{\gamma})$ is given by
\[
\|f\|_{1+k_{\gamma}}^2 =
|\hat{f}(0)|^2 + \sum_{h\neq 0} | \hat{f}(h)|^2 \,\frac{\omega_h}{\gamma},
\] 
which is again of similar type with modified weights for the 
Fourier coefficients.
\end{exmp}

\subsection{Functions of Finitely Many Variables}\label{s2.3}

First, we consider a single family of reproducing kernels $k_{\gamma}$ 
that is derived from a pair
of seminorms that 
satisfies
\ref{a2} and \ref{a3}.
Furthermore, we consider a sequence
$\bg=\left(\gamma_j\right)_{j\in \N}$ of positive 
weights.

For $s\in\N$
we define the reproducing kernel $K^{\bsgamma}_s$ on 
$D^s\times D^s$ by
\begin{align}\label{eq100}
K^{\bsgamma}_s(\bx,\by) = \prod_{j= 1}^s
(1+k_{\gamma_j}(x_j,y_j)), 
\end{align}
where $\bx, \by\in D^{s}$.

The reproducing kernel Hilbert space $H(K^{\bsgamma}_s)$
is the (Hilbert space) tensor product of the spaces $H(1+k_{\gamma_j})$. 
Thus Lemma \ref{lem10} implies that $H(K^{\bsgamma}_s)$,
as a vector space, does depend on the dimension $s$
and the vector space $H$ only.
We henceforth denote this vector space by $H_s$.

Now, we consider two families of reproducing kernels 
$k_{\gamma,\na}$ and 
$k_{\gamma,\nb}$ that are derived from two pairs 
of seminorms, both of which satisfy
\ref{a2} and \ref{a3} 
with a common space $H$.
We denote the reproducing kernel according to \eqref{eq100} by 
$K^{\bsgamma,\na}_s$ and  $K^{\bsgamma,\nb}_s$, respectively.

From the consideration above
we get that $H(K^{\bsgamma,\nn}_s)$,
as a vector space,
does neither depend on $\bsgamma$ nor on $\nn \in \{\na,\nb\}$.
For sequences $\bg$ and $\be$ of positive weights that differ
only by a multiplicative constant
we compare the norms $\|\cdot\|_{K^{\bsgamma,\na}_s}$
and $\|\cdot\|_{K^{\be,\nb}_s}$ on this vector space.
We use $\imath^{{\be},\bg}_s$ to denote the embedding of 
$H(K_s^{{\be},\nb})$ into $H(K_s^{\bg,\na})$.

For $c > 0$ we put 
$c \bg=\left(c \gamma_j\right)_{j\in \N}$.

\begin{theo}\label{theo1}
Let $0<c_0<1$ denote the constant according to Theorem~\ref{lem1}. 
For all sequences $\bg$ of positive weights 
\[
\left\|\imath^{c_0\bg,\bg}_s\right\| 
\leq \prod_{j=1}^s\left(1+\gamma_j\right)^{1/2}
\]
holds for the norm of the embedding of 
$H(K_s^{{c_0 \bg},\nb})$ into $H(K_s^{\bg,\na})$. 
\end{theo}

\begin{proof}
Note that
\[
\|\imath^{c_0\bg,\bg}_s\| 
\leq \prod_{j=1}^s \left\|\imath^{c_0\gamma_j,\gamma_j}_1\right\|,
\]
and employ Theorem~\ref{lem1}. 
\end{proof}

Observe that $\left(\imath^{\be,\bg}_s\right)^{-1}$ is the
embedding of $H(K^{\bg,\na}_s)$ into $H(K^{\be,\nb}_s)$.
Theorem~\ref{theo1} and the symmetry in our assumption 
yields the following result.

\begin{cor}\label{c1}
If
\begin{equation}\label{eq5}
\sum_{j=1}^\infty \gamma_j < \infty,
\end{equation}
then
\begin{equation}\label{eq6}
\sup_{s \in \N} \max \left\{ \|\imath^{c_0\bg,\bg}_s\|, \|
(\imath^{c_0^{-1} \bg,\bg}_s)^{-1}\|,
\|\imath^{\bg,c_0^{-1}\bg}_s\|, \|
(\imath^{\bg,c_0\bg}_s)^{-1}\| \right\} < \infty.
\end{equation}
\end{cor}

Under stronger assumptions, which are discussed in Remark \ref{r3}, 
a stronger conclusion is presented in \cite[Thm.~1]{HR13}.
We stress that the situation of Example~\ref{exa1} with $r \geq 2$
is not covered by \cite[Thm.~1]{HR13}, while
Theorem \ref{theo1} and Corollary \ref{c1} are of course applicable.

\begin{rem}
Consider the situation of Example~\ref{exa1} with $r=1$
as well as $\na = \text{A}$ and $\nb = \text{S}$. 
The following results follow from \cite[Exmp.~4, Exmp.~5, and 
Thm.~1]{HR13}.
We already get \eqref{eq6} from $\lim_{j \to \infty} \gamma_j = 0$,
so that \eqref{eq5} is not necessary for \eqref{eq6}
to hold. Actually \eqref{eq5} is equivalent to 
\begin{equation}\label{eq7}
\sup_{s \in \N} \max \left\{ \|\imath^{\bg,\bg}_s\|, \|
(\imath^{\bg,\bg}_s)^{-1}\|
\right\} < \infty
\end{equation}
in this case.
However, for $\na = \ \pitchfork$ and $\nb \in \{\text{A}, \text{S}\}$
we have equivalence of \eqref{eq7} and $\sum_{j=1}^\infty
\gamma_j^{1/2} < \infty$ and of \eqref{eq5} and \eqref{eq6}. 
In particular,
\eqref{eq5} does not imply \eqref{eq7} in general.
\end{rem}

\begin{rem}
\label{Remark5}
Let condition (\ref{eq5}) be satisfied.
Consider the situation of Examples~\ref{exa1} and \ref{exa2}
in the case $r\geq 2$. 

Firstly, let  $\nn=\text{S}$.
We already know that the pair of seminorms 
$\|\cdot\|_{1,\text{S}}$ and $\|\cdot\|_{2,\text{S}}$ satisfies
the assumptions 
\ref{a2} and \ref{a3}, 
while this is not true
for the pair of seminorms 
$\|\cdot\|_{1,\text{S}}$ and $\|\cdot\|_{2,\text{S}'}$.
We show that for these two pairs of seminorms an embedding result as in 
Corollary \ref{c1} no longer holds true.
More precisely, we demonstrate that there exists no constant $c_0 >0$ 
such that (\ref{eq6}) is satisfied.
Let 
\begin{align*}
f(\bx) =\prod_{j=1}^s f_j(x_j)
\end{align*}
with
\[
f_j(x_j) = \sqrt{3}\cdot x_j.
\]
For all sequences $\bsgamma$ of positive weights
we get 
\[
\prod_{j=1}^s \bigl(
\|f_j\|^2_{1,\text{S}} + \frac{1}{\gamma_j} \|f_j\|^2_{2,\text{S}'}
\bigr) =
1
\]
and
\[
\prod_{j=1}^s \Bigl(
\|f_j\|^2_{1,\text{S}} + \frac{1}{\gamma_j} \|f_j\|^2_{2,\text{S}}
\Bigr) =
\prod_{j=1}^s\Bigl(1+\frac{3}{\gamma_j}\Bigr),
\]
which diverges
for $s\to \infty$. 
Hence (\ref{eq6}) does not hold, regardless of how we choose $c_0>0$. 

Secondly, let $\nn\in\{\text{S},\pitchfork,\text{A}\}$. It is 
easily verified that the inequalities
\[
\prod^s_{j=1} \Bigl( \| f_j \|^2_{1, \nn'} + 
\frac{1}{\gamma_j}\| f_j \|^2_{2, \nn'} \Bigr) \le 6^s
\]
and 
\[
\prod^s_{j=1} \Bigl( \| f_j \|^2_{1, \nn} + 
\frac{1}{\gamma_j}\| f_j \|^2_{2, \nn} \Bigr) \ge 
\prod^s_{j=1} \frac{1}{\gamma_j}
\]
hold. This shows that for the pairs of seminorms 
$\|\cdot\|_{1,\nn'}$, $\|\cdot\|_{2,\nn'}$ and
$\|\cdot\|_{1,\nn}$, $\|\cdot\|_{2,\nn}$ there also exists no 
constant $c_0>0$ such that (\ref{eq6})
holds.
\end{rem}

\subsection{Functions of Infinitely Many Variables}\label{s2.4}

Again, we first consider a single family of reproducing kernels 
$k_{\gamma}$ that is 
derived from a pair
of seminorms that satisfies
\ref{a2} and \ref{a3}.
Furthermore, we consider a sequence
$\bg=\left(\gamma_j\right)_{j\in \N}$ of positive 
weights such that \eqref{eq5} is satisfied.

\begin{rem}\label{rem10}
Define the seminorms $\|\cdot\|_{1,\na}$ and $\|\cdot\|_{2,\na}$ 
on $H$ by
\begin{align*}
\|f\|_{1,\na} =|\langle f,1\rangle_{1}|
\end{align*}
and
\[
\|f\|_{2,\na} = \|f\|_{2}
\]
for all $f\in H$.
The assumptions \ref{a2} and \ref{a3} hold for the pair of
seminorms
$\|\cdot\|_{1,\na}$ and $\|\cdot\|_{2,\na}$, which fits
into the situation of Remark~\ref{rem1}.
We will employ this observation in several proofs, as it allows
to reduce the general setting to the particular setting of 
Remark~\ref{rem1}.
\end{rem}

The natural domain for the counterpart of \eqref{eq100} for
infinitely many variables is given by
\begin{align}\label{eq101}
\X^{\bg} = 
\Bigl\{{\bf x} \in D^{\mathbb{N}} \,\Big|\, \prod_{j=1}^\infty (1+
k_{\gamma_j}(x_j,  x_j) )<\infty \Bigr\}.
\end{align}
We present some basic properties of $\X^{\bg}$,
see also \cite{GMR12}.

\begin{lemma}\label{lem20}
Let $a,a_1,\dots,a_n\in D$. Then we have 
$(a_1,\dots,a_n,a,a,\dots)\in \X^{\bg}$,
and in particular $\X^{\bg} \neq\emptyset$.
Furthermore, we have $\X^{\bg} = \X^{c \bg}$
for all $c>0$.
\end{lemma}

\begin{proof}
First assume that we are in the situation of Remark~\ref{rem1}. Then 
there exists a reproducing kernel $k$ on $D\times D$ such 
that $1+k_{\gamma} =1+\gamma\cdot k$ for all 
$\gamma>0$. This implies ${\X^{\bg}} = {\X^{c \bg}}$ 
and furthermore 
$(a_1,a_2, \ldots, a_n, a,a, \ldots) \in {\X^{\bg}} \neq\emptyset$
due to the summability condition (\ref{eq5}).

Now consider the general case. 
For $\gamma>0$ denote by $k_{\gamma,\na}$ the  
reproducing kernel on $D\times D$
that corresponds to the pair of seminorms
$\|\cdot\|_{1,\na}$ and $\|\cdot\|_{2,\na}$,
see Remark~\ref{rem10},
and denote by $\X^{\bg,\na}$ the 
corresponding
subset of $D^\N$.
From Theorem~\ref{lem1}
we get
\begin{align*}
\frac{1}{1+\gamma}
\left(1+k_{{c_0}\gamma,\na}(x,x)\right)
\leq 1+k_{\gamma}(x,x)
\leq (1+{c_0^{-1}}\gamma)
\left(1+k_{{c_0^{-1} \gamma,\na}}(x,x)\right)
\end{align*}
for all $\gamma>0$ and $x\in D$.
Therefore
\begin{align*}
\X^{c_0^{-1} \bg,\na} \subseteq \X^{\bg} \subseteq \X^{c_0 \bg,\na}.
\end{align*}
Since we already know that
$(a_1,a_2, \ldots, a_n, a,a, \ldots) \in \X^{ \bg,\na} = 
\X^{c\bg,\na}$ for all $c>0$, the statement follows.
\end{proof}

We define the reproducing kernel $K^{\bsgamma}_\infty$ on
$\X^{\bsgamma}\times \X^{\bsgamma}$ by
\begin{align}\label{eq102}
K^{\bsgamma}_\infty({\bf x},{\bf y}) = 
\prod_{j=1}^\infty (1+k_{\gamma_j}(x_j, y_j))
\end{align}
for $\bx,\by \in \X^{\bg}$.
For a function $f\colon D^s \to \R$ and a set 
$\emptyset\neq\X\subseteq D^\N$ we define 
$\psi_s^{\X}f\colon \X \to \R$ by 
\begin{align}\label{g8}
\psi_s^{\X}f(\bx)
=f(x_1,\dots,x_s)
\end{align}
for $\bx\in\X$.

The following lemma is a generalization of \cite[Lem.~9]{HR13},
and it follows directly from Lemma~\ref{lemapp} from the Appendix.

\begin{lemma}\label{dicht}
The mapping $\psi_s^{\X^{\bg}}$ is a linear isometry from
$H(K_s^{\bg})$ into $H(K^{\bg}_\infty)$, 
and $\bigcup_{s\in \N}\psi_s^{\X^{\bg}}
H(K_s^{\bg})$ 
is a dense subspace of $H(K^{\bg}_\infty)$.
\end{lemma}

Now we consider two families of reproducing kernels 
$k_{\gamma,\na}$ and 
$k_{\gamma,\nb}$ that are derived from two pairs 
of seminorms, both of which satisfy
\ref{a2} and \ref{a3} 
with a common space $H$.

In the sequel, we let $\nn \in \{\na,\nb\}$.
We denote the set
according to \eqref{eq101}
by $\X^{\bsgamma,\nn}$
and the reproducing kernel according to \eqref{eq102}
by $K^{\bsgamma,\nn}_\infty$.

If $H(K^{{\be},\nb}_\infty) \subseteq H(K^{\bg,\na}_\infty)$, then
we use $\imath^{{\be},\bg}_\infty$ to denote the respective embedding.

\begin{theo}\label{theo2}
For 
every $c >0$ we have
\[
\X^{{c \bg},\nb} = \X^{\bg,\na}.
\]
Furthermore, with
$0<c_0<1$ denoting the constant according to Theorem~\ref{lem1},
\begin{align*}
H(K^{{c_0 \bg},\nb}_\infty) \subseteq H(K^{\bg,\na}_\infty),
\end{align*}
and
\[
\left\|\imath^{c_0\bg,\bg}_\infty\right\| 
\leq \prod_{j=1}^\infty \left(1+\gamma_j\right)^{1/2}
\]
holds for the norm of the respective embedding.
\end{theo}

\begin{proof}
As in the proof of Lemma \ref{lem20} we use Theorem \ref{lem1}
to derive
\[
1+k_{c_0 \gamma,\nb}(x,x)
\leq (1+\gamma)
\left(1+k_{\gamma,\na}(x,x)\right)
\]
for all $\gamma>0$ and $x \in D$.
We conclude that $\X^{\bg,\na} \subseteq \X^{c_0 \bg,\nb}$.
By the symmetry in our assumption and by Lemma~\ref{lem20} we get 
$\X^{{c \bg},\nb} = \X^{\bg,\na}$.

Put $\X = \X^{\bg,\na}$, $H_s = H(K^{\bg,\na}_s)$, and 
$H_0 = \bigcup_{s\in \N}\psi_s^{{\X}}(H_s)$.
Since $H_s = H(K_s^{c\bg,\nb})$,
Lemma~\ref{dicht} implies that $H_0$
is dense in $H(K^{c \bg,\nb}_\infty)$ and $H(K^{\bg,\na}_\infty)$. 
Furthermore, Theorem~\ref{theo1} and Lemma~\ref{dicht} imply
\begin{align*}
\|f\|_{K^{\bg,\na}_\infty}
\leq \prod_{j=1}^\infty\left(1+\gamma_j\right)^{1/2}
\cdot \|f\|_{K^{c_0 \bg,\nb}_\infty}
\end{align*}
for $f \in H_0$.
We conclude that $H(K^{{c_0 \bg},\nb}_\infty) \subseteq 
H(K^{\bg,\na}_\infty)$
with the norm of the embedding bounded as claimed.
\end{proof}

Theorem~\ref{theo2} and the symmetry in our assumption 
yield the following result.

\begin{cor}\label{c2}
We have
\[
\max \left\{ \|\imath^{c_0\bg,\bg}_\infty\|, \|
(\imath^{c_0^{-1}\bg,\bg}_\infty)^{-1} \|,
\|\imath^{\bg,c_0^{-1}\bg}_\infty\|, \|
(\imath^{\bg,c_0\bg}_\infty)^{-1}\| \right\} < \infty.
\]
\end{cor}

The comments and remarks that follow Corollary \ref{c1} carry over
to the present case of functions with infinitely many variables.
See, in particular, \cite[Thm.~2]{HR13}.

\section{The Integration Problem}\label{s3}

To begin with we derive some analytical properties
of the integration problem for functions with finitely many
and with infinitely many variables, see also \cite{GMR12}.
Then we present the framework for the analysis of the 
corresponding numerical integration problems and the basic 
application of the embedding results and norm estimates.

Again, we first consider a single family of reproducing kernels 
$k_{\gamma}$ that is derived from a pair
of seminorms satisfying
\ref{a2} and \ref{a3}.
Furthermore, we consider a sequence
$\bg=\left(\gamma_j\right)_{j\in \N}$ of positive 
weights such that \eqref{eq5} is satisfied.

Additionally we assume that $\rho$ is a probability measure 
(on a given $\sigma$-algebra) on $D$ such that
\begin{align*}
H \subseteq L^1(D,\rho).
\end{align*}
We let $\rho^s$ and $\rho^\N$ denote the corresponding product
measures on (the product $\sigma$-algebras in)
$D^s$ and $D^\N$, respectively. 
For $s \in \N$ we put $1:s = \{1,\dots,s\}$.

\subsection{Analytic Properties}

\begin{lemma}\label{integraluniformly}
We have
\begin{align*}
H(K_s^\bg)\subseteq L^1(D^s,\rho^{s})
\end{align*}
for all $s\in\N$.
The respective embeddings $J_s$ are continuous, and 
\begin{align*}
\sup_{s\in\N}\|J_s\|
<\infty.
\end{align*}
\end{lemma}

\begin{proof}
We use the same strategy as in the proof of Lemma~\ref{lem20}. 
Using Theorem~\ref{theo1} and Remark~\ref{rem10} 
we may, without loss of generality, assume that we are in the 
situation of Remark~\ref{rem1}, i.e.,
there exists a reproducing kernel $k$ on $D\times D$ such 
that $1+k_{\gamma} =1+\gamma\cdot k$ for all 
$\gamma>0$.

For $u\subseteq 1:s$ we define
the reproducing kernel $k_u$ on $D^s$ by
\begin{equation}\label{ku}
k_u({\bf x},{\bf y}) = 
\prod_{j\in u} k(x_j, y_j)
\end{equation}
for $\bx,\by \in D^s$ as well as
\begin{align*}
\gamma_u
=\prod_{j\in u} \gamma_j.
\end{align*}
Note that
\begin{align*}
K^\bg_s
=\sum_{u\subseteq 1:s} \gamma_u k_u.
\end{align*}

From Lemma~\ref{lem10} we get $H(1) \cap H(k) = \{0\}$.
As a well known consequence,
$H(K_s^\bg)$ is the direct sum of the spaces $H(\gamma_u k_u)$,
i.e., for every $f\in H(K_s^\bg)$
there exist uniquely determined 
$f_u\in H(k_u)$
such that
\begin{align}\label{eq103}
f
=\sum_{u\subseteq 1:s} f_u,
\end{align}
and in this case we get
\begin{equation}\label{sumofnorms}
\|f\|^2_{K^\bg_s}
=\sum_{u\subseteq 1:s} \|f_u\|^2_{\gamma_u k_u}
=\sum_{u\subseteq 1:s} \frac{1}{\gamma_u}\|f_u\|^2_{k_u}.
\end{equation}
See, e.g., \cite[Prop.\ 1 and Lem.\ 11]{GMR12}.

The closed graph theorem implies that $H(k)$ is continuously embedded
into $L^1(D,\rho)$.
Let $d$ denote the norm of this embedding, multiplied by 
$\sqrt{\pi/2}$.
Use Lemma~\ref{l100} from the Appendix to conclude that
$H(k_u)\subseteq L^1(D^s,\rho^{s})$ and
\begin{align*}
\int_{D^{s}} \left| f_u\right|\,\mathrm d\rho^{s}
\leq d^{\left|u\right|} \|f_u\|_{k_u}
\end{align*}
for all $f_u\in H(k_u)$ and all $u\subseteq 1:s$.
For 
$f\in H(K_s^\bg)$  
of the form \eqref{eq103} with $f_u\in H(k_u)$
this yields 
\begin{align*}
\int_{D^{s}} \left| f\right|\,\mathrm d\rho^{s}
& 
\leq \sum_{u\subseteq 1:s}\int_{D^{s}} 
\left| f_u\right|\,\mathrm d\rho^{s}
\leq \sum_{u\subseteq 1:s} d^{\left|u\right|} \|f_u\|_{k_u}\\
& 
\leq 
\Bigl(\sum_{u\subseteq 1:s} d^{2\left|u\right|}\gamma_u\Bigr)^{1/2}
\cdot \|f\|_{K_s^\bg}.
\end{align*}
Due to \eqref{eq5}, this shows the claim.
\end{proof}

Define the linear functional $I_s\colon H(K_s^\bg)\to \R$ by
\begin{align*}
I_s(f)
=\int_{D^s}f\,{\rm d}\rho^{s}
\end{align*}
for all $f\in H(K_s^\bg)$. 

\begin{rem}\label{r10}
Note that $\|I_s\| \geq 1$, since $I_s(1) = 1$ and $\|1\|_{K^\bg_s}=1$.
Furthermore, $\|I_s\|\leq \|J_s\|$,
and hence Lemma~\ref{integraluniformly} leads to
\begin{align}\label{eq104}
1 \leq \sup_{s\in\N}\|I_s\|
<\infty.
\end{align}
\end{rem}

Recall that $H_s$ denotes the vector space $H(K_s^\bg)$.

\begin{lemma}\label{lemintegral}
There exists a uniquely determined bounded linear functional
\begin{align*}
       I_\infty \colon H(K^\bg_\infty)\to\R
\end{align*}
such that
\begin{align*}
I_\infty (\psi_s^{\X^{\bg}}(f))
=I_s(f)
\end{align*}
for all $f\in H_s$ and $s\in\N$.
\end{lemma}

\begin{proof}
This follows directly from Lemma~\ref{dicht} and \eqref{eq104}.
\end{proof}

If $\X^\bg$ is measurable, $\rho^\N(\X^\bg)=1$, and 
$H(K^\bg_\infty)\subseteq L^1(\X^\bg, \rho^\N)$ then 
Lemma~\ref{lemintegral} yields
\begin{align*}
I_\infty (f)
=\int_{\X^\bg} f\,\mathrm d\rho^\N
\end{align*}
for all $f\in H(K^\bg_\infty)$. 
For sufficient conditions under which these assumptions are fulfilled 
we refer to \cite{GMR12}.

Note that every function $f\colon \X^\bg\to \R$ with 
$f\in H(K^\bg_\infty)$ is measurable 
(with respect to the trace of the product $\sigma$-algebra
in $D^\N$).
This follows directly from Lemma~\ref{dicht}, 
Lemma~\ref{integraluniformly}, and the fact that the pointwise limit 
of measurable functions is measurable again.

\subsection{Algorithms and Minimal Errors}\label{s3.2}

Let 
\[
s \in \N \cup \{\infty\}.
\]
The integration problem on $H(K^\bg_s)$ consists in the 
approximation of $I_s$ by deterministic or randomized
algorithms. The corresponding domain of the integrands
is given by
\begin{equation*}
\X_s =\begin{cases}
\, D^s,
\hspace{2ex}&\text{if $s \in\N$},\\
\, \X^\bg,
\hspace{2ex} &\text{if $s = \infty$.}
\end{cases}
\end{equation*} 
We confine ourselves to deterministic and randomized linear algorithms 
of the form
\begin{equation}\label{Qn}
Q(f) = \sum^n_{i=1} w_i f(\bst^{(i)}).
\end{equation}
For a deterministic linear algorithm $Q$, the number $n \in \N$ of 
knots, and the knots $\bst^{(i)} \in \X_s$ as well as the coefficients
$w_i\in \R$ are fixed, regardless of $f$, 
i.e., $Q$ is a quadrature formula.
The corresponding class of algorithms is denoted by
$\A^{\det}_{s}$.
For a randomized linear algorithm $Q$, the number of knots $n\in\N$ is 
fixed as previously, but now the knots $\bst^{(i)}$ and coefficients $w_i$
are random variables with values in $\X_s$ and $\R$, respectively.
Any such algorithm is a mapping 
$Q\colon H(K^\bg_s) \times \Omega \to \R$ such that 
$Q(\cdot,\omega) \in \A^{\det}_{s}$ for every $\omega \in \Omega$,
where $(\Omega,\Sigma,P)$ denotes the underlying probability space.
The corresponding class of algorithms is denoted by $\A^{\ran}_{s}$.
We stress that we only consider non-adaptive algorithms, i.e., the 
random (or deterministic) choice of the coefficients and the knots is
independent of the specific integrand $f$. Actually, we will only need 
algorithms of this type to establish our upper error 
bounds. The lower bounds for the error that we will present 
usually hold for much more 
general classes of algorithms, see our comments below.

The error $e(Q,f)$ of approximating the integral $I_s(f)$ for 
$f\in H(K^\bg_s)$
by a randomized or deterministic linear algorithm $Q$ is defined as
\begin{displaymath}
e(Q, f) = \Bigl( \EE \bigl( \left( I_s(f) - Q(f) \right)^2 \bigr) 
\Bigr)^{1/2}. 
\end{displaymath}
Clearly $\A^{\det}_{s} \subsetneq \A^{\ran}_{s}$, and 
for deterministic algorithms $Q$ the error simplifies to 
\begin{displaymath}
e(Q,f) =  \left| I_s(f) - Q(f) \right|.
\end{displaymath}
The worst case error $e(Q,K^\bg_s)$ of approximating 
the integration functional $I_s$ on the unit ball
in $H(K^\bg_s)$ by $Q$ is defined as
\begin{displaymath}
e(Q,K^\bg_s) = \sup \left\{
e(Q,f) \mid f\in H(K^\bg_s),\ \|f\|_{K^\bg_s} \leq 1\right\}.
\end{displaymath}

Let
\begin{align}\label{modfunc}
{\rm cost}\colon \A^{\rm ran}_{s}\to [1,\infty]
\end{align}
be a given function that assigns to every randomized (or
deterministic) linear algorithm its cost.
The key quantity in the analysis
is the $n$th minimal error given by
\[
e^{\rm set}_{{\rm cost}}(n,K^\bg_s) = \inf\{e(Q,K^\bg_s) \,|\, 
Q \in \mathcal{A}^{\rm set}_{s}\text{ with }{\rm cost}(Q)\leq n\},
\]
where ${\rm set}\in\{ {\rm ran}, {\rm det}\}$.

\subsection{Application of the Norm Estimates}

Let $\|\cdot\|_{1,\na}$ and $\|\cdot\|_{2,\na}$ as well as 
$\|\cdot\|_{1,\nb}$ and $\|\cdot\|_{2,\nb}$ be two pairs 
of seminorms on $H$, both satisfying 
\ref{a2} and \ref{a3}.
The norm estimates from Sections \ref{s2.3} and \ref{s2.4}
are applied in the analysis of the integration problem as follows.
The linearity of $Q$ or $Q(\cdot,\omega)$ from \eqref{Qn}, 
respectively, ensures that
$e(Q,c f) = |c| \cdot e(Q,f)$ for every $c \in \R$.
From Corollaries~\ref{c1} and \ref{c2} we get the following
result.

\begin{theo}\label{t3}
Let ${\rm set}\in\{ {\rm ran}, {\rm det}\}$ and 
let $0<c_0<1$ denote the constant according to Theorem~\ref{lem1}. 
For every sequence $\bg$ of weights that satisfies \eqref{eq5}
there exists a constant $c \geq 1$ with the following property.
For every $s \in \N \cup \{\infty\}$ and every
$Q \in \A^{\rm set}_{s}$ we get
\[
c^{-1} \cdot e\left(Q, K^{c_0 \bg,\nb}_s\right) \leq
e\left(Q, K^{\bg,\na}_s\right) \leq 
c \cdot e\left(Q, K^{c_0^{-1} \bg,\nb}_s\right).
\]
In particular, for every cost function \eqref{modfunc},
every $s \in \N \cup \{\infty\}$, and every $n \in \N$
we get
\begin{align*}
c^{-1} \cdot e^{\rm set}_{{\rm cost}} \left( n,K^{c_0 \bg,\nb}_s \right)
\leq e^{\rm set}_{{\rm cost}} \left( n,K^{\bg,\na}_s \right)
\leq c \cdot e^{\rm set}_{{\rm cost}} 
\left( n,K^{c_0^{-1} \bg,\nb}_s \right).
\end{align*} 
\end{theo}

Analogous estimates are valid, of course,
for other linear problems like approximation
(recovery) of functions.

Ultimately, we wish to transfer results
from the spaces $H(K^{\bg,\na}_s)$, say, to the spaces 
$H(K^{\bg,\nb}_s)$.
According to Theorem \ref{t3}, switching
from one pair of seminorms to the other one
involves a compensation by a suitable multiplication of the 
weights. However, for a single pair of seminorms
$\|\cdot\|_{1}$ and $\|\cdot\|_{2}$ we observe the following
strong impact of this multiplication, as far as embeddings
are concerned.
Let $c > 1$. 
For $s \in \N$ we have equivalence of the
norms on $H(K^{c \bg}_s)$ and $H(K^{\bg}_s)$, but
the norm of the embedding of $H(K^{c \bg}_s)$ into
$H(K^{\bg}_s)$ grows exponentially with $s$. For 
$s=\infty$ we do not even have 
$H(K^{c \bg}_s) \subseteq H(K^{\bg}_s)$.

The transfer of results is still possible
for integration problems that share the following feature:
Results and constructions depend on the underlying sequence $\bg$
of weights only via their decay, 
\begin{equation*}
\decay(\bg) = \sup \Bigl(
\Bigl\{ p > 0 \, \Big|\, 
\sum^\infty_{j=1} \gamma_j^{1/p} < \infty \Bigr\} \cup \{0\}
\Bigr).
\end{equation*}
Given this feature, it remains to observe that
$\decay(c \bg) = \decay(\bg)$ for every $c>0$.
Details and important examples will be presented in the following two 
sections.

\section{Results on Multivariate Integration}\label{RMI}

Consider the setting from Section \ref{s3} for 
\[
s \in \N.
\]
We employ the standard cost function
\[
\std \colon \A^{\rm ran}_{s}\to[1,\infty)
\]
for finite-dimensional integration that simply counts
the number of function evaluations, i.e.,
$\std(Q) = n$ for any algorithm of the form \eqref{Qn}.

In tractability analysis one studies the behavior of the $n$th
minimal error simultaneously in $n$ and $s$. A key concept is
strong polynomial tractability, i.e., the existence
of $c, \alpha > 0$ such that
\[
\forall\, n,s \in \N: 
e^{\rm set}_{\std}(n,K^\bg_s)
\leq c \cdot n^{-\alpha}.
\]
Obviously one is interested in the largest such $\alpha$, which leads 
to the definition
\begin{equation*}
\lambda^{\rm set}_{\std}= 
\lambda^{\rm set}_{\std}((K^\bg_s)_{s \in \N})= 
\sup \Bigl(\Bigl\{ \alpha > 0 \,\bigg|\, \sup_{s,n\in \N} 
e_{\std}^{\rm set}(n,K^\bg_s) \cdot  
n^\alpha < \infty \Bigr\}\cup \{0\}\Bigr).
\end{equation*}
We add that the normalized error is studied in many papers, 
i.e., $e_{\std}^{\rm set}(n,K^\bg_s)$ is replaced by 
$e_{\std}^{\rm set}(n,K^\bg_s)/ \|I_s\|$.
In our situation both concepts coincide,
as far as the strong polynomial tractability is concerned,
see Remark~\ref{r10}.
Moreover, we add that 
$1/\lambda^{\rm set}_{\std}$
is called the exponent of strong tractability. 
For more information about tractability
we refer to the monograph series
\cite{NW08, NW10, NW12}.

The generic application of our embedding results to
multivariate integration is the following, straightforward
consequence of Theorem \ref{t3}.

\begin{cor}\label{Cor_Trans_Princ}
Let ${\rm set} \in \{\det,\ran\}$ and suppose that 
$\lambda_{\std}^{\rm set}((K^{\bg, \na}_s)_{s \in \N})$ depends 
on $\bg$ only via $\decay(\bg)$, i.e., for all summable sequences
$\bg,\be$ of weights with $\decay(\bg) = \decay(\be)$
we have
\[
\lambda_{\std}^{\rm set}((K^{\bg, \na}_s)_{s \in \N}) =
\lambda_{\std}^{\rm set}((K^{\be,\na}_s)_{s \in \N}).
\]
Then we  have
\[
\lambda_{\std}^{\rm set}((K^{\bg,\na}_s)_{s \in \N}) =
\lambda_{\std}^{\rm set}((K^{\bg,\nb}_s)_{s \in \N}).
\] 
\end{cor}

A similar conclusion can obviously be drawn for upper and lower 
bounds on $\lambda_{\std}^{{\rm set}}$ that only depend on the 
decay of the weights.
 
In this section we apply the transfer principle from Corollary 
\ref{Cor_Trans_Princ} to $s$-fold weighted tensor 
products of the Sobolev space $W^{r,2}[0,1]$ of smoothness $r$ as 
spaces of integrands, see  Example \ref{exa1}, and to the uniform 
distribution $\rho$ on $[0,1]$.

For this kind of spaces quasi-Monte Carlo (QMC) theory was
known, so far, to provide very good deterministic algorithms 
in the anchored case
$H( K_s^{\bg, \pitchfork})$ and very good randomized algorithms 
in the ANOVA case $ H( K_s^{\bg, A})$,
as we will explain in more detail below. 
For more background on quasi-Monte Carlo integration we refer to the 
recent survey article \cite{DKS13}.

\subsection{Deterministic Setting}

For the anchored Sobolev spaces $H( K_s^{\bg, \pitchfork})$ of any
smoothness $r\in\N$ there are very good QMC algorithms known. In fact,
for $r=1$ there are efficient lattice rules or $(t,m,s)$-nets available, 
see, e.g., \cite{DKPS05, Kuo03, NC06a, Wan03} and the literature 
mentioned therein. For $r\geq 2$ one may use polynomial lattice rules of 
higher order, see, e.g., \cite{ BDGP11, BDLNP12, DG12, DP07} and the 
literature mentioned therein, or the recently analyzed interlaced 
polynomial lattice rules, see \cite{God15}.

We state the main result of this section, which was
partially known already. More precisely,
the upper bound for $\lambda_{\std}^{\rm det}$ follows directly
from the well-known lower bound for the minimal error
for univariate integration on $W^{r,2}[0,1]$ in the deterministic setting,
see \cite[Prop.~1, Sec.~1.3.12]{Nov88}.
In fact, this result holds for all deterministic algorithms.
Furthermore, for $r = 1$ the lower bound was known before for all cases 
$*\in\{\text{S}, \text{A}, \pitchfork\}$, 
see \cite[Cor.~6]{Kuo03} and \cite[Thm.~3]{SWW04}, 
while for $r\ge 2$ it was known before only
for $\nn= \ \pitchfork$, see \cite[Sec.~5.4]{DG12}.

\begin{theo}\label{CorDet}
Let $H=W^{r,2}[0,1]$ and $*\in\{ {\rm S}, {\rm A}, \pitchfork\}$. 
Then we have
\[
\min \left\{ \frac{\decay(\bg)}{2}, r \right\} \le 
\lambda_{\std}^{\rm det} ((K^{\bg,\nn}_s)_{s \in \N})      \le r.
\]
\end{theo}

\begin{proof}
Since the upper bound 
for $\lambda_{\std}^{\rm det}$
is already known, 
we only provide the proof of the lower bound. 
We first consider the anchored setting $\nn = \ \pitchfork$.
To this end, let $r\ge 1$ be an integer, let $b$ be a prime
and let $1/2 \leq \tau < r$. Then there exists a constant
$C_{r,\tau,b}> 0$ such that for all $m, s \in \N$
there exists a deterministic linear algorithm 
$Q_{m} \in \A^{\det}_{s}$ with cost 
$\std(Q_{m}) = b^m$ and 
\begin{equation}\label{det_pol_latt}
e(Q_{m}, K^{\bg,\pitchfork}_s)  \le (b^m-1)^{-\tau} \;  
\prod_{j=1}^s \left(1+C_{r,\tau,b} \cdot \gamma_j^{1/(2\tau)}
\right)^{\tau},
\end{equation}
see \cite[Thm.~5.3]{DKPS05} for $r=1$ and
\cite[Sec.~5.4]{DG12} for $r \geq 2$. Clearly 
\[
\prod_{j=1}^\infty 
\left(1+C_{r,\tau,b} \cdot \gamma_j^{1/(2\tau)} \right)^{\tau} <
\infty
\hspace{3ex}\text{if and only if}\hspace{3ex} 
\sum_{j=1}^\infty \gamma_j^{1/(2\tau)} < \infty,
\] 
and the latter holds if $\tau < \decay(\bg)/2$ or $\tau = 1/2$. 
Hence the lower bound 
for $\lambda_{\std}^{\rm det}$
holds in the case $\nn = \ \pitchfork$.

Next we consider $\nn \in \{{\rm S}, {\rm A}\}$. From Example 
\ref{exa1} we know that Theorem \ref{t3} is applicable for 
$\na= \ \pitchfork$ and $\nb \in \{{\rm S}, {\rm A}\}$.
Since $\decay(c_0     \bg) = \decay(\bg)$, we obtain the lower bound
for $\lambda_{\std}^{\rm det}$
also in the case $\nn  \in \{{\rm S}, {\rm A}\}$.
\end{proof}

\begin{rem}
If $\decay(\bg)<2r$ the upper and lower bound for 
$\lambda_{\std}^{\rm det}$ in Theorem~\ref{CorDet} do not coincide.
For the case $r=1$ it is conjectured that actually the lower bound
is sharp, see \cite[Open~Problem~72]{NW10}.
\end{rem}

\begin{rem}\label{Rem_Pol_Lat}
The algorithms used to derive the error bound (\ref{det_pol_latt}) 
in \cite{DKPS05,DG12} 
belong to the class of (shifted) polynomial lattice rules. 
Polynomial lattice rules are actually not lattice rules, but QMC-cubature 
rules whose integration points belong to a special family
of $(t,m,s)$-nets. They were introduced by Niederreiter in \cite{N92a}. 
To ensure favorable error bounds, sometimes a shift 
$\boldsymbol{\sigma} \in [0,1)^s$ has to be added to the 
integration points of a polynomial lattice rule $Q$, where the addition 
is meant component-wise modulo 1. The resulting QMC-cubature 
$Q(\boldsymbol{\sigma})$ is then called a shifted polynomial lattice rule.

For a given prime base $b$ an $s$-dimensional polynomial lattice rule 
is constructed with the help of a generating vector 
$\boldsymbol{q}$ whose entries $q_1, \ldots, q_s$ are polynomials over 
the finite field of order $b$. 

For smoothness $r=1$ or for slowly decaying weights
it is sufficient to consider classical
polynomial lattice rules with a shift. In \cite{DKPS05} rules of this 
type  that satisfy (\ref{det_pol_latt}) in the corresponding regime 
$1/2 \le \tau <1$
were constructed by means of a
component-by-component (cbc) algorithm. To exploit higher smoothness 
$r\ge 2$, higher-order polynomial lattice rules
were introduced by Dick and Pillichshammer in \cite{DP07}. In 
\cite[Sec.~5.4]{DG12} the error bound (\ref{det_pol_latt}) was derived 
for $1\le \tau <r$ and higher-order polynomial lattice rules,
without a shift, by utilizing 
\cite[Thm.~3.1]{BDGP11}.

The polynomial lattice rules used to derive (\ref{det_pol_latt}) can be 
constructed by cbc algorithms, based on the fast cbc algorithm from 
\cite{NC06a, NC06b}, requiring $O(rsn^r \ln(n))$ operations and  
$O(n^r)$ memory, see \cite{BDLNP12}. However, it is not known how to 
determine a proper shift, if needed, in a efficient way.
Explicit formulas for the constant $C_{r,\tau,b}$ in 
(\ref{det_pol_latt}) for $1\le \tau <r$ as well as for 
$1/2 \le \tau <1$ can be found in \cite[Sec.~5.4]{DG12}.

We close this remark by mentioning that in the recent paper \cite{God15} 
interlaced polynomial lattice rules have been analyzed that 
serve the same purpose
as higher-order polynomial lattice rules, 
but can be constructed with a cbc algorithm that only requires 
$O(rs n \ln(n))$ operations and $O(n)$ memory.
\end{rem}

\begin{rem} 
The lower bound for $\lambda_{\std}^{{\rm set}}$ in Theorem \ref{CorDet}
in the anchored setting $\nn = \ \pitchfork$
is  a direct consequence of the upper error bound 
(\ref{det_pol_latt}). The key point in the proof of (\ref{det_pol_latt})
in \cite[Sec.~5.3 and 5.4]{DG12} is to bound certain $s$-dimensional 
Walsh norms by $s$-dimensional anchored norms to make use of 
\cite[Thm.~3.1]{BDGP11} and the results from \cite{DKPS05}.
These norm bounds can be established with the help of a multivariate 
Taylor expansion of the integrand in $(a,\ldots,a) \in [0,1]^s$, where 
$a\in [0,1]$ is the anchor 
that defines the norm according to Example \ref{exa1}.
It is crucial that the anchored norm is perfectly suited to work with 
Taylor expansions. This is, e.g., not the case for ANOVA norms,
and therefore it would be elaborate to try to prove the lower
bound from Theorem \ref{CorDet} 
directly for $* = {\rm A}$ without making use of our Theorem \ref{t3}.
\end{rem}

\begin{rem}\label{r44}
Let $\tau < \min \{ \decay(\bg)/2,r\}$.
According to the proof of Theorem \ref{CorDet} and Remark
\ref{Rem_Pol_Lat} there exists a constant $c>0$ with the following
property for every $s \in \N$. For every $m \in \N$ 
a (shifted) polynomial lattice rule $Q_m$ with $b^m$ points in
$[0,1]^s$ is available such that
$e(Q_{m}, K^{\bg,\pitchfork}_s)  \le c \, (b^m-1)^{-\tau}$.
If we consider $\nn \in \{{\rm S}, {\rm A}\}$ 
instead of $\nn = \ \pitchfork$, then the same algorithms $Q_m$
satisfy the same error bound, up to a possibly different
constant $c$ that again does not depend on $s$, see Theorem
\ref{t3}.

These findings carry over to the randomized algorithms
discussed in Remark \ref{Interlaced},
if we take $\nn = {\rm A}$ as the starting
point and then consider $\nn \in \{{\rm S}, \pitchfork\}$.
\end{rem}

\subsection{Randomized Setting}

In the randomized setting several very good QMC algorithms are
known for the integration problem on the ANOVA spaces
$H( K_s^{\bg, \text{A}})$. For $r=1$ one may take the scrambled 
Niederreiter nets analyzed in \cite{YH05} or the scrambled polynomial 
lattice rules analyzed in \cite{BD11}, and for $r\ \geq 2$ one may
use the interlaced scrambled polynomial lattice rules considered in
\cite{GD12, DG13}.

So far, however, there were no good randomized algorithms
known for the anchored spaces $H( K_s^{\bg, \pitchfork})$, 
and this holds true even for $r=1$.

For instance, Hickernell et al.~\cite{HMNR10} have studied
infinite-dimensional integration in the anchored setting for $r=1$,
using single- and multilevel algorithms with
classical Monte Carlo methods as finite-dimensional building blocks. 
To achieve better results for infinite-dimensional integration,
they have asked for finite-dimensional integration algorithms on 
$H( K_s^{\bg, \pitchfork})$  superior to Monte Carlo, see the last 
sentences in \cite[Sec.~4.3 and 5.3]{HMNR10}.

With the help of the results from Section~\ref{SEC_EMB} we can deduce 
that all linear algorithms that perform well on the ANOVA space also 
perform well on the anchored Sobolev space. In particular, we are able to 
present randomized QMC algorithms, namely interlaced scrambled polynomial 
lattice rules, which outperform classical Monte Carlo algorithms
substantially.

We state the main result of this section, which was partially known 
already. More precisely, the upper bound for $\lambda_{\std}^{\rm ran}$
follows directly from the well-known lower bound for the
minimal error for univariate 
integration on $W^{r,2}[0,1]$ in the randomized setting, see 
\cite[Prop.~1(ii), Sec.~2.2.9]{Nov88}. 
In fact, this result holds for all randomized algorithms.
The lower bound was known before only in the case $\nn =
{\rm A}$; for $r=1$ we refer to \cite{BD11} and for $r \geq 2$ we
refer to \cite[Thm.~5.1]{DG13}, which makes use of \cite[Thm.~1]{GD12}.

\begin{theo}\label{CorRan}
Let $H=W^{r,2}[0,1]$ and 
$*\in\{ {\rm S}, {\rm A}, \pitchfork\}$. Then we have
\[
\min \left\{ \frac{\decay(\bg)}{2}, r+ \frac{1}{2} \right\} 
\le \lambda_{\std}^{\rm ran}((K^{\bg,*}_s)_{s\in\N}) \le r+\frac{1}{2}.
\]
\end{theo}

\begin{proof}
Since the upper bound 
for $\lambda_{\std}^{\rm ran}$
is already known, we only 
provide the proof of the lower bound. 
We first consider the ANOVA setting $\nn = {\rm A}$.
To this end, let $r\ge 1$ be an integer, let $b$ be a prime,
and let $1/2 \leq \tau < r +1/2$. Then there exists a constant
$C_{r,\tau,b}> 0$ such that for all $m, s \in \N$
there exists an unbiased randomized linear algorithm $Q_{m} \in
\A^{\ran}_{s}$ 
with cost ${\std}(Q_{m})=b^m$ and
\begin{equation}\label{ran_pol_latt}
e(Q_{m}, f)^2 
\le (b^{ m}-1)^{-2\tau} \;  
\bigg[ C_{r,\tau,b} \prod_{j=1}^s \left(1+C_{r,\tau,b} \cdot 
\gamma_j^{1/(2\tau)} \right) \bigg]^{2\tau} \|f\|^2_{K^{\bg, {\rm A}}_s}
\end{equation}
for all $f\in H(K^{\bg,{\rm A}}_s)$, see \cite[Thm.~5.1]{DG13}.
We proceed as in the proof of Theorem \ref{CorDet}
to derive the lower bound 
for $\lambda_{\std}^{\rm ran}$
in the case $\nn = {\rm A}$
and to extend this result to the case $\nn  \in \{{\rm S}, \pitchfork\}$.
\end{proof}

\begin{rem}\label{Interlaced}
The algorithms used to derive the error bound (\ref{ran_pol_latt}) belong 
to the class of interlaced scrambled polynomial lattice rules. 
For a given prime base $b$ an interlaced scrambled polynomial lattice 
rule $Q$ of order $r$ consisting of $b^m$ points in dimension $s$ is 
constructed in the following way: First an ordinary polynomial lattice 
rule with $b^m$ points in dimension $rs$ is generated and afterwards 
the points are randomized via Owen's $b$-ary digit scrambling \cite{O95}. 
Then each of the resulting $rs$-dimensional points 
$x=(x_1,x_2,\ldots, x_{rs})$ is mapped to an $s$-dimensional point
\begin{equation*}
\big( \mathcal{D}_r(x_1,\ldots,x_r), 
\mathcal{D}_r(x_{r+1},\ldots,x_{2r}),\ldots, 
\mathcal{D}_r(x_{r(s-1)+1},\ldots,x_{rs}) \big)
\end{equation*} 
by applying the digit interlacing function
\begin{equation*}
\mathcal{D}_r: [0,1)^r \to [0,1)\,, \,
(y_1,\ldots,y_r) \mapsto 
\sum_{i=1}^\infty \sum^r_{j=1} y_{j,i} b^{-j-(i-1)r},
\end{equation*}
where $y_j = y_{j,1}b^{-1} + y_{j,2}b^{-2} + \ldots $ for $1\le j \le r$.
Interlacing is important to achieve the higher order convergence 
rate $r+1/2$ for $r\ge 2$; notice that for $r=1$ we have that 
$\mathcal{D}_r$ is the identity mapping on $[0,1)$ and consequently 
``interlaced scrambled polynomial lattice rules of order $1$'' are 
nothing but ordinary scrambled polynomial lattice rules. 
For more details see, e.g., \cite{DG13, GD12}.

Notice that Owen's scrambling procedure implies that each point of 
the resulting interlaced scrambled polynomial lattice rule $Q$ is
uniformly distributed on $[0,1)^s$. Hence $Q$ is unbiased, 
implying 
\[
e(Q,f)^2 = \Var(Q(f))
\]
for every integrand $f$.

Interlaced polynomial lattice rules were introduced by Dick in \cite{D11a}.
In \cite{GD12} it was shown that for product weights the construction 
cost of the component-by-component (cbc) algorithm that generates the 
interlaced polynomial lattices rules, based on the fast cbc algorithm 
from \cite{NC06a, NC06b}, is of order $O(rsmb^m)$ operations using $
O(b^m)$ memory. 

These cubature rules settle the question in \cite[p.~245]{HMNR10} for 
good randomized algorithms for finite-dimensional integration 
in the anchored Sobolev space. In particular, they can be employed 
to establish error bounds for infinite-dimensional integration 
in the fixed subspace sampling model that improve the corresponding 
results in \cite{HMNR10} substantially, see Section \ref{FSS}.
\end{rem}

\begin{rem}
The lower bound for $\lambda_{\std}^{\rm ran}$ in the ANOVA setting is 
a direct consequence of the upper error bound (\ref{ran_pol_latt}),
which actually is a bound on the variance of $Q_{m}(f)$,
since $Q_{m}$ is unbiased for every $f: D^s \to \R$.
To analyze the variance, it is convenient to use the ANOVA%
\footnote{ANOVA  is an acronym for ``analysis of variance''.} 
decomposition of $Q_{m}(f)$. Due to our specific randomization 
it turns out that
\begin{equation*}
\left[ Q_{m}(f) \right]_u = Q_{m}(f_u)
\end{equation*}
for all $u\in 1:s$;
here $\left[ Q_{m}(f) \right]_u$ denotes the $u$th ANOVA component 
of $Q_{m}(f)$ with respect to the randomness induced by the 
scrambling procedure and $f_u\in H(k_u)$ denotes the $u$th ANOVA 
component of $f$ with respect to $\rho^s$, cf.~(\ref{ku}) and 
(\ref{sumofnorms}).
A rigorous formulation of this ``ANOVA invariance principle'' can
be found in \cite[Lem.~2.1]{BG12}.
Hence 
\begin{equation*}
\Var \left( Q_{m}(f) \right) = 
\sum_{u\subseteq 1:s} \Var \left( Q_{m}(f_u) \right).
\end{equation*}
This fact and the identity
\begin{equation*}
\|f\|^2_{K^{\bg,{\rm A}}_s} = 
\sum_{u\subseteq 1:s} \gamma_u^{-1} \|f_u\|_{k_u}^2,
\end{equation*}
cf. (\ref{sumofnorms}), are essential in the analysis of the 
cbc algorithm that generates the integration rules $Q_{m}$.

It is not clear to the authors how to prove the lower bound for 
$\lambda_{\std}^{\rm ran}$ in the anchored setting directly 
without using Theorem \ref{t3}.
\end{rem}

\begin{rem}
A basic issue is to decide whether randomized algorithms are
superior to deterministic algorithms, i.e., whether
$\lambda^{{\rm ran}}_{\std} > \lambda^{{\rm det}}_{\std}$.
For a positive answer a lower bound on $\lambda^{{\rm ran}}_{\std}$ and 
an upper bound on $\lambda^{{\rm det}}_{\std}$ is needed, and the
converse is needed for a negative answer.
Due to Theorems \ref{CorDet} and \ref{CorRan}
the superiority holds true if
\[
\decay(\bg) > 2r.
\]
\end{rem}

\section{Results for Infinite-Dimensional Integration}\label{RIDI}

Consider the setting from Section \ref{s3} for
\begin{align*}
s=\infty.
\end{align*}
As in the previous section,  
anchored spaces are much more suited for a direct analysis
of deterministic algorithms, while ANOVA spaces are much more suited
for a direct analysis of randomized algorithms, see the discussion
of the literature below. Embeddings and norm estimates allow
to transfer the respective results.

\subsection{Cost Models for Infinite-Dimensional Integration}

In contrast to finite-dimensional integration,
the choice of an appropriate cost model is an issue
in the present setting. Here we do not have a canonical cost model
anymore, but the models that are studied in the literature
share the following feature. The cost of a single function
evaluation at a point $\bst$ is no longer independent of $\bst \in D^\N$,
and not even uniformly bounded in $\bst$.

We present three such cost models. All of them only account for 
the cost of function evaluations, as in the finite-dimensional case, and
they employ a nondecreasing function $\$\colon \N_0 \to [1,\infty)$ and 
a ``default value'' $a\in D$.

Let
\[
\T_u = \{ \bst \in D^\N \mid \text{$\bst_j = a$ for all $j \in \N
\setminus u$}\}
\]
for any finite subset $u \subseteq \N$, where $\bst_j$ denotes the 
$j$th component of 
$\bst$, and let $Q\in \A^{{\rm ran}}_{\infty}$ denote any 
randomized linear algorithm of the form \eqref{Qn}.
In the sequel, we use the convention $\min \emptyset=\infty$.

Fixed subspace sampling basically means that all function 
evaluations of an algorithm have to take place in a set $\T_{1:s}$ with a 
fixed value of $s$, and the same cost $\$(s)$ is assigned to any such 
evaluation. In the fixed subspace sampling model the cost of $Q$ is 
therefore given by
\[
{\rm fix}(Q)
= n \cdot
\min\{ \$(s) \mid 
\text{$s\in\N_0$ such that 
$\bst^{(1)}(\omega) , \dots, \bst^{(n)}(\omega) 
\in \T_{1:s}$ for all $\omega \in \Omega$} \}.
\]
This model directly corresponds to the classical
approach to infinite-dimensional integration,
namely the approximation by an $s$-dimensional
integration problem by setting all variables
with indices $j > s$ to the default value $a$.

In the two other models all function evaluations of
an algorithm take place in the set $\bigcup_{s \in \N} \T_{1:s}$.
The cost for an evaluation at a point $\bst$ from this set
is either determined by the maximal value of $j$ such that
$\bst_j \neq a$ or by the number of components of $\bst$ that
are different from $a$.

In the nested subspace sampling model the cost of $Q$ is given by
\begin{align*}
{{\rm nest}}(Q)
= \sum_{i=1}^n
\min\{ \$(s) \mid 
\text{$s\in\N_0$ such that 
$\bst^{(i)}(\omega) \in \T_{1:s}$ for all $\omega \in \Omega$} \}.
\end{align*}
This model was introduced in \cite{CDMR09}
in a more general setting, and it was actually called 
``variable subspace sampling model''. We prefer the name 
``nested subspace sampling model'' to clearly distinguish this model 
from the cost model we present next.
In the unrestricted subspace sampling model the cost of $Q$ is given by
\begin{align*}
{{\rm unr}}(Q)
=\sum_{i=1}^n
\min\{ \$(|u|) \mid 
\text{$u \subseteq \N$ finite such that 
$\bst^{(i)}(\omega) \in \T_u$ for all $\omega \in \Omega$} \}.
\end{align*}
The unrestricted subspace sampling model was introduced in 
\cite{KSWW10a} (where it did not get a specific name).
In the definition of the three cost functions 
a certain property is required to hold for all $\omega \in
\Omega$. Often this worst case point of view is 
replaced by an average case. We stress that such a replacement would 
not affect the cost of the algorithms that we use to establish our 
upper bounds for the $n$th minimal errors.
However, to streamline the presentation we consider the 
worst case.

Obviously ${{\rm unr}}(Q) \leq {{\rm nest}}(Q) \leq {{\rm fix}}(Q)$,
so that the corresponding minimal errors satisfy
\begin{align}\label{triv1}
e^{{\rm set}}_{{\rm unr}}(n,K^\bg_\infty)
\leq e^{{\rm set}}_{{\rm nest}}(n,K^\bg_\infty)
\leq e^{{\rm set}}_{{\rm fix}}(n,K^\bg_\infty)
\end{align}
for ${\rm set} \in \{\ran, \det\}$.
Furthermore, 
\begin{align}\label{triv2}
e^{{\rm ran}}_{{\rm cost}}(n,K^\bg_\infty)
\leq e^{{\rm det}}_{{\rm cost}}(n,K^\bg_\infty),
\end{align}
where ${\rm cost} \in \{ {\rm fix}, {\rm nest}, {\rm unr}\}$. 
In order to simplify the presentation we put
\begin{align}\label{lambdacost}
\lambda^{{\rm set}}_{{\rm cost}}=
\lambda^{{\rm set}}_{{\rm cost}}(K^\bg_\infty)
=\sup \Bigl\{ \alpha \ge 0 \mid \sup_{n\in \N} 
e^{\rm set}_{{\rm cost}}(n,K^\bg_\infty) \cdot  n^\alpha < \infty
\Bigr\}.
\end{align}
The inequalities \eqref{triv1} and \eqref{triv2} directly yield
\begin{align*}
\lambda^{{\rm set}}_{{\rm fix}}
\leq \lambda^{{\rm set}}_{{\rm nest}}
\leq \lambda^{{\rm set}}_{{\rm unr}}
\qquad
\text{and}
\qquad
\lambda^{{\rm det}}_{{\rm cost}}
\leq \lambda^{{\rm ran}}_{{\rm cost}}.
\end{align*}

\subsection{General Results}

Let $\|\cdot\|_{1,\na}$ and $\|\cdot\|_{2,\na}$ as well as
$\|\cdot\|_{1,\nb}$ and $\|\cdot\|_{2,\nb}$ be two pairs 
of seminorms on $H$, both satisfying 
\ref{a2} and \ref{a3}.
For infinite-dimensional integration we have a similar 
transfer principle as for multivariate integration.

\begin{cor}\label{Principle_infty}
Let ${\rm cost} \in \{ {\rm fix}, {\rm nest}, {\rm unr}\}$ 
and ${\rm set} \in \{\ran, \det\}$.
Suppose that  
$\lambda^{\rm set}_{{\rm cost}}(K^{\bg, \na}_\infty)$ depends on $\bg$ only
via $\decay(\bg)$. 
Then we get 
\[
\lambda^{\rm set}_{{\rm cost}}(K^{\bg,\na}_\infty) =
\lambda^{\rm set}_{{\rm cost}}(K^{\bg,\nb}_\infty).
\]
\end{cor}
 
Since multiplication of the weights $\bg$ by a positive constant does 
not affect $\decay(\bg)$, Corollary \ref{Principle_infty} follows 
directly from Theorem \ref{t3}.  Obviously, a similar conclusion can
be drawn for upper and lower bounds on $\lambda^{{\rm set}}_{\rm cost}$ 
that only depend on the decay of the weights.

In the deterministic setting we have another general result,
which deals with a single pair 
of seminorms on $H$ that satisfies \ref{a2} and \ref{a3} and
with the corresponding kernels $1+k_1$ and $K^{\bg}_\infty$.
A second pair of such seminorms is only employed in the proof.
We say that a kernel $k:D\times D\to \R$ is anchored, if
there exists an $a\in D$ with $k(a,\cdot) = 0$. 
In the setting of Remark~\ref{r0} with an anchored kernel
$k$, the following theorem was already known before,
see \cite[Thm.~2]{PW11}. More precisely, 
for this particular case,
the lower bound on
$\lambda^{{\rm det}}_{\unr}$ was established in \cite{PW11} and the
upper bound can be derived from the analysis in \cite[Sec.~3.3]{KSWW10a}.
In analogy to \eqref{lambdacost} we define
\begin{align*}
\lambda^{\rm det}_{\std}(1+k_1)
=\sup \Bigl(
\Bigl\{ \alpha \ge 0 \,\Big|\, \sup_{n\in \N} 
e^{\rm det}_{{\std}}(n,1+k_1) \cdot  n^\alpha < \infty \Bigr\}
\Bigr),
\end{align*}
which deals with one-dimensional integration on the Hilbert space
$H=H(1+k_1)$, equipped with the norm $\|\cdot\|_H=\|\cdot\|_{1+k_1}$.

\begin{theo}\label{Theo_UB_PW}
If the cost function $\$$ satisfies $\$(\nu) = \Omega(\nu)$ and 
$\$(\nu) = O(e^{\sigma \nu})$ for some $\sigma\in (0,\infty)$,
then we have
\begin{equation*}
\lambda^{{\rm det}}_{\unr}(K^{\bg}_\infty) =
\min \left\{ \lambda^{{\rm det}}_{\std}(1+k_1),\,
\frac{{\rm decay}(\bg) - 1}{2} \right\}.
\end{equation*}
\end{theo}

\begin{proof}
We choose an arbitrary $a\in D$ and consider the bounded linear 
functional $\xi$ on $H$ that is given by $\xi(f)=f(a)$ 
for each $f\in H$. We define a new pair of seminorms on this space
by 
\[
\|f\|_{1,\nb} = |\xi(f)|
\]
and 
\[
\|f\|_{2,\nb} = \|f-\xi(f)\|_{H},
\]
see Remark \ref{r3a}. 
We have $\lambda^{{\rm det}}(1+k_1) = 
\lambda^{{\rm det}}(1+k_{1, \nb})$,
since $\|\cdot \|_{H}$ and $\|\cdot\|_{1+k_{1,\nb}}$ 
are equivalent norms.
Moreover $k_{1,\nb}(a,\cdot)=0$ according to Remark \ref{rem1}.
Notice that the assumptions of \cite[Thm.~2]{PW11} are fulfilled
for the space $H(K^{\bg,\nb}_\infty)$.
Indeed, the algorithms that satisfy \cite[Eqn.~(10)]{PW11} 
can be obtained from univariate linear quadrature rules with
convergence rates arbitrarily close to 
$\lambda^{{\rm det}}(1+k_1)$
with the help of Smolyak's construction, see \cite[Sec.~3.3]{PW11}.
Now \cite[Thm.~2]{PW11} ensures that the statement of 
Theorem \ref{Theo_UB_PW} holds for $H(K^{\bg,\nb}_\infty)$. 
Due to Corollary \ref{Principle_infty} it thus also holds for 
$H(K^{\bg}_\infty)$. 
\end{proof}

\begin{rem}\label{r33}
The algorithm used to establish the lower bound on 
$\lambda^{{\rm det}}_{\unr}$ in Theorem \ref{Theo_UB_PW}
is a multivariate decomposition method, formerly known as 
changing dimension algorithm. This type of algorithm
was introduced and
analyzed in \cite{KSWW10a} and the analysis was refined in \cite{PW11}. 
This analysis crucially relies
on the anchored decomposition in the space of integrands. 
According to Theorem \ref{Theo_UB_PW} the 
multivariate decomposition methods is applicable far beyond the anchored
setting. This also applies to the algorithms used to establish 
the lower bounds on $\lambda^{{\rm set}}_{\unr}$ in Theorem~\ref{Unr_Sob}.
\end{rem}

\subsection{Tensor Products of Weighted Sobolev Spaces}

Now we turn to the particular
case of $\infty$-fold weighted tensor products
$H(K^{\bg, *}_\infty)$ of the Sobolev spaces $W^{r,2}[0,1]$ of 
smoothness $r\in\N$,  where $*\in\{ {\rm S}, {\rm A}, \pitchfork\}$, see
Example~\ref{exa1}. 

\subsubsection{Unrestricted Subspace Sampling}

In the anchored case $\nn= \ \pitchfork$ the
statement of the next theorem was known in the deterministic
setting, see \cite[Thm.~2, Sec.~3.3]{PW11} and \cite[Sec.~3.3]{KSWW10a}.
The upper bound on
$\lambda^{\ran}_{\unr}(K^{\bg,\pitchfork}_\infty)$ was  known
before, see \cite[Sec.~3.2.1]{Gne12a}, and the lower bound for 
$\lambda^{\ran}_{\unr}(K^{\bg, \pitchfork}_\infty)$ was known for 
$r=1$, see \cite[Exmp.~2]{PW11}. 
In the ANOVA case $\nn = {\rm A}$ the statement of the Theorem was
known in the randomized setting for arbitrary $r\in\N$, see
\cite[Cor.~5.3]{DG13}, 
where the algorithms from \cite{PW11} were modified
for the ANOVA setting.
We add that the same result was 
proved in \cite{DG13} for the class of finite-intersection weights. 
Furthermore, the upper bound on  
$\lambda^{\ran}_{\unr}(K^{\bg, {\rm A}}_\infty)$ verified 
in \cite{DG13} holds for much more general randomized algorithms 
than for the linear algorithms of the form (\ref{Qn}).
See also Table \ref{my-label1} for an overview of known
and new results.

\begin{theo}\label{Unr_Sob}
Let  $\$$ satisfy $\$(\nu) = \Omega(\nu)$ and 
$\$(\nu) = O(e^{\sigma \nu})$ for some $\sigma\in (0,\infty)$. 
Moreover, let $H=W^{r,2}[0,1]$ and 
$*\in\{ {\rm S}, {\rm A}, \pitchfork\}$. Then we have
\begin{equation*}
\lambda^{\det}_{\unr}(K^{\bg,*}_\infty) = 
\min \left\{ r, \frac{\decay(\bg) -1}{2} \right\}
\end{equation*}
and 
\begin{equation*}
\lambda^{\ran}_{\unr}(K^{\bg,*}_\infty) = 
\min \left\{ r + \frac{1}{2}, \frac{\decay(\bg) -1}{2} \right\}.
\end{equation*}
\end{theo}

\begin{proof}
Since the result is true in the deterministic setting for 
$\nn=\ \pitchfork$ and in the randomized setting for $\nn={\rm A}$, and 
since $\lambda^{\det}_{\unr}(K^{\bg, \pitchfork}_\infty)$ and 
$\lambda^{\ran}_{\unr}(K^{\bg, {\rm A}}_\infty)$ depend on the 
weights $\bg$ only via $\decay(\bg)$, it remains to apply 
Corollary \ref{Principle_infty}.
\end{proof}

\begin{table}[ht]
\centering
\caption{Matching upper and lower bounds for 
$\lambda_{\rm unr}^{\rm set}$}
\label{my-label1}
\begin{tabular}{|c|c|c|c|c|c|}
\hline
\multicolumn{2}{|c|}{\multirow{2}{*}{${\rm cost}={\rm unr}$}} & \multicolumn{2}{c|}{$r=1$}                          & \multicolumn{2}{c|}{$r>1$}                          \\ \cline{3-6} 
\multicolumn{2}{|c|}{}                                        & $\det$               & $\ran$                       & $\det$               & $\ran$                       \\ \hline
\multirow{2}{*}{$*=\ \pitchfork$}        & lower bound        & \multicolumn{2}{c|}{\cite{PW11}}                    & \cite{PW11}          & new                          \\ \cline{2-6} 
                                         & upper bound        & \cite{KSWW10a}       & \cite{Gne12a}                & \cite{KSWW10a}       & \cite{Gne12a}                \\ \hline
\multirow{2}{*}{$*=\text{A}$}            & lower bound        & \multirow{2}{*}{new} & \multirow{2}{*}{\cite{DG13}} & \multirow{2}{*}{new} & \multirow{2}{*}{\cite{DG13}} \\ \cline{2-2}
                                         & upper bound        &                      &                              &                      &                              \\ \hline
\multirow{2}{*}{$*=\text{S}$}            & lower bound        & \multirow{2}{*}{new} & \cite{DG13}                  & \multirow{2}{*}{new} & \cite{DG13}                  \\ \cline{2-2} \cline{4-4} \cline{6-6} 
                                         & upper bound        &                      & new                          &                      & new                          \\ \hline
\end{tabular}
\end{table}

\begin{rem}
In the situation of Theorem~\ref{Unr_Sob} we may choose one sequence of 
algorithms that achieves the optimal convergence rate for all three 
cases $*\in\{ {\rm S}, {\rm A}, \pitchfork\}$ simultaneously.
More precisely,
there exists a sequence $(Q_n)_{n\in\N}$ of deterministic or
randomized algorithms, respectively, with
${{\rm unr}}(Q_n)\leq n$
with the following property for every 
$0<\lambda<\lambda^{\rm set}_{\unr}$.
There exists a constant $c>0$ such that
\[
e\left(Q_n, K^{\bg,*}_\infty\right)
\leq c\cdot n^{-\lambda}
\]
for all $*\in\{ {\rm S}, {\rm A}, \pitchfork\}$
and $n\in\N$.
Such a sequence can be obtained by the following procedure.
Choose any of the three norms, say, the anchored norm,
and let $c_0>0$ be the minimum of the value of 
the constant according to Theorem~\ref{lem1} for the two pairs 
$(\pitchfork, {\rm A})$ and $(\pitchfork, {\rm S})$.
Now, take a sequence of algorithms according to Theorem~\ref{Unr_Sob} 
for the sequence $c_0^{-1}\bg$ of weights and $*=\ \pitchfork$.
This sequence has the desired property, since
$H(K^{\bg,*}_\infty) \subseteq 
H(K^{c\cdot \bg,*}_\infty)$ for all $c\geq 1$ and  
$*\in\{ {\rm S}, {\rm A}, \pitchfork\}$.

A similar remark also applies to Theorems~\ref{Nest_Det}--\ref{ThmFix}.
\end{rem}

\begin{rem}\label{Rem_Ran_Det}
Due to Theorem \ref{Unr_Sob} 
randomized algorithms are superior to deterministic algorithms for
unrestricted subspace sampling, i.e., 
$\lambda^{{\rm ran}}_{\unr} > \lambda^{{\rm det}}_{\unr}$,
if and only if 
\[
\decay(\bg) > 2r + 1.
\]
Observe that the known results have only covered the case
$r=1$ and $*= \ \pitchfork$ and, partially, the case
$r\geq 2$ and $*= \ \pitchfork$.
\end{rem}

\subsubsection{Nested Subspace Sampling}

Here we start with the analysis of deterministic algorithms.
In the anchored case $*= \ \pitchfork$ the statement of the 
next Theorem was already known. 
More precisely,
the lower bound on $\lambda^{\det}_{\nes}(K^{\bg, \pitchfork}_\infty)$ 
was established in \cite[Sec.~5]{DG12} by using multilevel algorithms 
based on the higher-order polynomial lattice rules described 
in Remark \ref{Rem_Pol_Lat}. This lower bound improves on earlier 
results in \cite{Gne10, NHMR11}.
The upper bound on $\lambda^{\det}_{\nes}(K^{\bg,
\pitchfork}_\infty)$ can be derived easily from
\cite[Thm.~4]{NHMR11} and is explicitly stated in \cite[Cor.~4]{DG12}. 
For the other cases 
$* \in \{{\rm S}, {\rm A}\}$ the statement of the theorem is new.
We add that the result for the case $\nn = \ \pitchfork$ does not
only hold for the class of product weights but also for the larger
class of product and order dependent (POD) weights and for the
class of finite intersection weights, see \cite[Cor.~4 and~6]{DG12}.

\begin{theo}\label{Nest_Det}
Let $\sigma\ge (2r-1)/2r$ and let the cost function 
$\$$ satisfy $\$(\nu) = \Theta(\nu^\sigma)$.  
Moreover, let $H=W^{r,2}[0,1]$ and 
$*\in\{ {\rm S}, {\rm A}, \pitchfork\}$. Then we have
\begin{equation*}
\lambda^{\det}_{\nes}(K^{\bg,*}_\infty) = 
\min \left\{ r, \frac{\decay(\bg) -1}{2\sigma} \right\}.
\end{equation*}
\end{theo}

\begin{proof}
Since the result is true for $\nn=\ \pitchfork$ and since
$\lambda^{\det}_{\nes}$ depends on the weights $\bg$ only via 
$\decay(\bg)$, the theorem follows from 
Corollary \ref{Principle_infty}.
\end{proof}

Now we consider randomized algorithms.
In the ANOVA case $\nn= {\rm A}$ the upper bound on
$\lambda^{\ran}_{\nes}(K^{\bg, {\rm A}}_\infty)$ was known and the
lower bound on $\lambda^{\ran}_{\nes}(K^{\bg, {\rm A}}_\infty)$ was
known for $r=1$, see \cite[Cor.~3.5 and~5.5]{BG12}. In the anchored
case only the upper bound on $\lambda^{\ran}_{\nes}(K^{\bg,
\pitchfork}_\infty)$ was known, see \cite[Sec.~3.2.1]{Gne12a}.
See also Table \ref{my-label2} for an overview of known and
new results with matching upper and lower bounds, covering both, 
deterministic and randomized algorithms.

\begin{theo}\label{Nest_Ran}
Let $\sigma\ge 2r/(2r+1)$ and let the cost function $\$$ satisfy 
$\$(\nu) = \Theta(\nu^\sigma)$. Moreover, let $H=W^{r,2}[0,1]$ and 
$\nn\in\{ {\rm S}, {\rm A}, \pitchfork\}$. Then we have
\begin{equation*}
\min \left\{ \max\left\{ r,\frac{3}{2}\right\}, 
\frac{\decay(\bg) -1}{2\sigma} \right\}
\leq
\lambda^{\ran}_{\nes}(K^{\bg,*}_\infty) \le 
\min \left\{ r + \frac{1}{2}, \frac{\decay(\bg) -1}{2\sigma} \right\}
\end{equation*}
with equality for $r=1$ or $\decay(\bg) \leq 2 \sigma r+1$.
\end{theo}

\begin{proof}
The upper bound is true for $\nn \in \{\pitchfork, {\rm A}\}$,
and it remains to apply Theorem \ref{t3} to also establish this bound
for $\nn = {\rm S}$. Starting with $\nn = {\rm A}$,
we establish the lower bound for $r=1$ in the same way. 

It remains to prove the lower bound for $r \geq 2$. Here we
start with $\nn = {\rm S}$, and
we use $H_s^{(r)}$ to denote the Hilbert space $H (K^{\bg,{\rm S}}_s)$ 
with its dependence on $r$. For $s \in \N$ the norm of the embedding of 
$H_s^{(r)}$ into $H_s^{(1)}$ obviously is one. 
To address the case $s=\infty$ we first observe that
$\X^{\bg,\pitchfork} = \X^{\bg, {\rm A}} = D^\N$, which
follows from Remark \ref{rem1} and the boundedness of the
corresponding kernels $k$.
Theorem \ref{theo2} implies that $\X^{\bg,{\rm S}} = D^\N$.
Use \cite[Prop.~2]{GMR12} to conclude that
$H_\infty^{(r)}$ is embedded into $H_\infty^{(1)}$ with norm
one as well.
Hence the lower bound for $r=1$ is also valid for $r \geq 2$.
This result is transferred to the case $\nn \in \{\pitchfork, {\rm
A}\}$ in the standard way.
Of course, the lower bound from Theorem \ref{Nest_Det}, which deals
with deterministic algorithms, is applicable, too.
\end{proof}

\begin{table}[ht]
\centering
\caption{Matching upper and lower bounds for 
$\lambda_{\rm nest}^{\rm set}$}
\label{my-label2}
\begin{tabular}{|c|c|c|c|c|c|}
\hline
\multicolumn{2}{|c|}{\multirow{2}{*}{${\rm cost}={\rm nest}$}} & \multicolumn{2}{c|}{$r=1$}                          & \multicolumn{2}{c|}{$r>1$}           \\ \cline{3-6} 
\multicolumn{2}{|c|}{}                                         & $\det$               & $\ran$                       & $\det$               & $\ran$        \\ \hline
\multirow{2}{*}{$*=\ \pitchfork$}         & lower bound        &
\cite{DG12}          & new                          & \cite{DG12}
& \multirow{2}{*}{open}          \\ \cline{2-5} 
                                          & upper bound        & \cite{NHMR11}        & \cite{Gne12a}                & \cite{NHMR11}        &  \\ \hline
\multirow{2}{*}{$*=\text{A}$}             & lower bound        &
\multirow{2}{*}{new} & \multirow{2}{*}{\cite{BG12}} &
\multirow{2}{*}{new} & \multirow{2}{*}{open}       \\ \cline{2-2}  
                                          & upper bound        &                      &                              &                      &    \\ \hline
\multirow{2}{*}{$*=\text{S}$}             & lower bound        &
\multirow{2}{*}{new} & \cite{BG12}                  &
\multirow{2}{*}{new} & \multirow{2}{*}{open}          \\ \cline{2-2} \cline{4-4} 
                                          & upper bound        &                      & new                          &                      &            \\ \hline
\end{tabular}
\end{table}

\begin{rem}
Relying on a modification of the component-by-component (cbc)
algorithm presented in \cite{BD11}, scrambled polynomial lattice 
rules were constructed in \cite[Section~5]{BG12} that were used as 
building blocks of multilevel algorithms to establish the lower 
bound on $\lambda^{\ran}_{\nes}$ for $r=1$ in the ANOVA setting. These 
scrambled polynomial lattice rules cannot exploit higher order smoothness. 
For $r\ge 2$ the authors believe that a matching lower bound for 
$\lambda^{\ran}_{\nes}(K^{\bg, {\rm A}}_\infty)$ can be established
with multilevel algorithms as in \cite{BG12} that use as building 
blocks the interlaced polynomial lattice rules from 
Remark~\ref{Interlaced}; a proof of this claim is nevertheless 
beyond the scope of the present paper. 

Notice that  the lower bound on $\lambda^{\ran}_{\nes}$ in the
anchored case is new and shows in particular that the corresponding
upper bound established in \cite[Sec.~3.2.1]{Gne12a} was already optimal. 
\end{rem}

\begin{rem}
To compare the power of deterministic and randomized algorithms
for nested subspace sampling we consider the condition
\begin{equation}\label{g88}
\decay(\bg) > 2\sigma r + 1,
\end{equation}
and we apply Theorems \ref{Nest_Det} and \ref{Nest_Ran}. For $r=1$ 
randomized algorithms are superior to deterministic algorithms,
i.e., $\lambda^{{\rm ran}}_{\nes} > \lambda^{{\rm det}}_{\nes}$,
if and only if \eqref{g88} holds true.
For $r \geq 2$ we only know that
\eqref{g88} is a necessary condition for
superiority.
\end{rem}

\subsubsection{Fixed Subspace Sampling in the Randomized 
Setting}\label{FSS}

Here we use the interlaced 
scrambled polynomial lattice rules  for multivariate 
integration described in Remark \ref{Interlaced} to improve and
generalize the results on fixed subspace sampling from 
\cite[Sec.~4.3]{HMNR10}. As in \cite{HMNR10} we focus on the randomized 
setting and the cost function $\$$ given by $\$(k)=k$ for all $k\in\N_0$.

\begin{theo}\label{ThmFix}
Let $H= W^{r,2}[0,1]$ and  
$*\in \{ \text{S}, \text{A}, \pitchfork \}$. Moreover,
let 
\[
\beta = \frac{1}{2} \min\{ \decay(\bg), 2r+1 \}.
\]
Then 
\begin{equation*}
\frac{\beta (\decay(\bg) -1)}{2\beta -1 + \decay(\bg)} \le 
\lambda_{{\rm fix}}^{{\rm ran}}(K^{\bg,*}_\infty)
\le \frac{(r+\frac{1}{2}) (\decay(\bg) -1)}{2r + \decay(\bg)}
\end{equation*}
with equality for $\decay(\bg) \ge 2r+1$.
\end{theo}

\begin{proof}
Combine the well-known lower bound for the minimal error for 
univariate integration on $W^{r,2}[0,1]$ in the randomized setting, see 
\cite[Prop.~1(ii), Sec.~2.2.9]{Nov88}, and \cite[Thm.~2]{HMNR10} to 
derive the upper bound
for $\lambda_{{\rm fix}}^{{\rm ran}}(K^{\bg,*}_\infty)$.
Combine the lower bound from Theorem \ref{CorRan}
and \cite[Thm.~1]{HMNR10} to derive the lower bound
for $\lambda_{{\rm fix}}^{{\rm ran}}(K^{\bg,*}_\infty)$ in the
case $\nn \in \{\pitchfork,{\rm A}\}$, and thus also for $\nn =
{\rm S}$.
\end{proof}

\begin{rem}\label{RemFix}
We compare the lower bound from Theorem \ref{ThmFix} with the
result from \cite[Sec~4.3]{HMNR10}. Since in \cite[Sec.~4.3]{HMNR10} 
only the case $r=1$ is treated, let us confine ourselves to this case.
For $* = \text{A}$ our 
lower bound recovers the lower bound from
\cite[Cor.~3.1]{Bal12}, which relies on scrambled polynomial lattice
rules analyzed in \cite{BD11} and which improved on the bound from
\cite[Cor.~1]{HMNR10}. The latter relies on scrambled Niederreiter 
$(t,m,s)$-nets analyzed in \cite{YH05}. 
In the case $*= \ \pitchfork$ our lower bound improves substantially 
on the lower bound 
\begin{equation*}
\frac{\decay(\bg) - 1}{2\decay(\bg)} \le \lambda_{{\rm fix}}^{\ran}
\end{equation*}
from \cite[Cor.~2]{HMNR10}.
It also improves on the better lower bound
\begin{equation}\label{ksww10}
\lambda_{{\rm fix}}^{\ran} \ge \begin{cases}
\, \decay(\bg)(\decay(\bg) -1)/(4\decay(\bg) - 2),
\hspace{2ex}&\text{if $1 < \decay(\bg) < 2$},\\
\, (\decay(\bg) -1)/(\decay(\bg) + 1),
\hspace{2ex} &\text{if $\decay(\bg) \ge 2$.}
\end{cases}
\end{equation}
that is mentioned in \cite[Rem.~4]{HMNR10} and that was achieved in  
\cite{KSWW10a} with the help of deterministic algorithms. More
precisely, for $1< \decay(\bg) \le 2$ both lower bounds coincide,
but for $\decay(\bg) > 2$ our lower bound is strictly better than
(\ref{ksww10}). This settles the open problem at the end 
of \cite[Rem.~4]{HMNR10}, where the authors asked whether 
(\ref{ksww10}) can be improved if one uses randomized algorithms 
different to classical Monte Carlo algorithms.
\end{rem}

\subsection{Tensor Products of Weighted Korobov Spaces}

Finally, we turn to spaces of periodic functions.
The weighted Korobov spaces $H(K^\bg_s)$ are tensor products of 
the periodic Sobolev spaces $H^r[0,1]$ of smoothness $r > 1/2$, see 
Example \ref{e33}. Tractability results for integration on 
$H(K^\bg_s)$ with $s \in \N$ have been established, e.g.,
in \cite{HW00,KPW14,SW01}. Here we consider the case $s = \infty$
of $\infty$-fold weighted tensor products and 
unrestricted subspace sampling.

For the periodic Sobolev spaces we have
$k_\gamma = \gamma \cdot k_1$ with $k_1$ not being an anchored
kernel, so that the results from \cite{KSWW10a,PW11}
on the multivariate decomposition method are not directly applicable.
Instead, one may study embeddings of the weighted
Korobov spaces $H(K^\bg_\infty)$ into $\infty$-fold weighted
tensor products of the non-periodic Sobolev spaces $W^{r,2}[0,1]$
(with fractional smoothness). More conveniently,
Theorem \ref{Theo_UB_PW} immediately yields the following 
result.

\begin{theo}
If the cost function $\$$ satisfies $\$(\nu) = \Omega(\nu)$ and 
$\$(\nu) = O(e^{\sigma \nu})$ for some $\sigma\in (0,\infty)$,
then we have
\begin{equation*}
\lambda^{{\rm det}}_{\unr}(K^{\bg}_\infty) =
\min \left\{ r,\, \frac{{\rm decay}(\bg) - 1}{2} \right\}.
\end{equation*}
\end{theo}

\section{Appendix}

\subsection{A Dense Subspace of $H(K)$}

Consider a sequence $(k_j)_{j \in \N}$ of reproducing kernels on $D
\times D$ for some set $D\neq\emptyset$ such that
\begin{align*}
H(1) \cap H(k_j) =\{0\}
\end{align*}
for every $j \in \N$.
For $s\in\N$ we define the reproducing kernel $K_s$ on $D^s\times D^s$ by
\begin{align*}
K_s(\bx,\by) = \prod_{j= 1}^s (1 + k_j(x_j,y_j)),
\end{align*}
where $\bx, \by\in D^{s}$.
In the sequel we assume that
\begin{align*}
\X
=\{\bx\in D^\N \mid 
\prod_{j=1}^\infty(1+k_j(x_j,x_j))<\infty\}\neq\emptyset,
\end{align*}
and we define the reproducing kernel $K$ on $\X\times\X$ by
\begin{align*}
K(\bx,\by) = \prod_{j=1}^\infty (1+k_{j}(x_j, y_j)),
\end{align*}
where $\bx,\by \in\X$.

The following lemma and its proof are almost identical to 
\cite[Prop.\ 2.18]{HEF14}.
Recall the definition of $\psi_s^\X$ in \eqref{g8}.

\begin{lemma}\label{lemapp}
The mapping $\psi_s^{\X}$ is a linear isometry from 
$H(K_s)$ into $H(K)$, and $\bigcup_{s\in \N}\psi_s^{\X}(H(K_s))$ 
is a dense subspace of $H(K)$.
\end{lemma}

\begin{proof}
First we show that
\begin{equation}\label{onein}
1\in H(K)
\end{equation}
with
\begin{equation}\label{oneone}
\|1\|_{K}=1.
\end{equation}
Note that
\begin{align*}
\X
=D^s\times R_s,
\end{align*}
where
\begin{align*}
R_s
=\{\bx\in D^{\{s+1,s+2,\dots\}} \mid 
\prod_{j=s+1}^\infty(1+k_j(x_j,x_j))<\infty\}.
\end{align*}
Consider the reproducing kernel $K^s$ on $R_s \times R_s$, given by
\[
K^s(\bx,\by)=\prod_{j=s+1}^\infty(1+k_j(x_j,y_j))
\]
for $\bx,\by \in R_s$. Obviously,
\[
K=K_s\otimes K^s,
\]
and we have $1\in H(K_s)$ with $\|1\|_{K_s}=1$.
Fix $\by\in \X$, and define $f_s : \X \to \R$ by
\[
f_s(\bx)=K^s((x_{s+1},x_{s+2},\dots),(y_{s+1},y_{s+2},\dots))
\]
for $\bx\in \X$. It follows that
$\lim_{s \to \infty} f_s(\bx)=1$ for all $\bx\in\X$.
Furthermore, $f_s\in H(K)$ with
\[
\|f_s\|_{K}^2
=\prod_{j=s+1}^\infty(1+k_j(y_j,y_j)),
\]
so that $\lim_{s \to \infty} \|f_s\|_{K} = 1$.
Similarly, we obtain
\begin{align*}
\langle f_{s_1},f_{s_2}\rangle_{K}
=\prod_{j=s_2+1}^\infty(1+k_j(y_j,y_j))
=\|f_{s_2}\|_{K}^2
\end{align*}
for $1\leq s_1\leq s_2$, which yields
\begin{align*}
\|f_{s_1}-f_{s_2}\|_{K}^2
=\|f_{s_1}\|_{K}^2+\|f_{s_2}\|_{K}^2-2
\langle f_{s_1},f_{s_2}\rangle_{K}
=\|f_{s_1}\|_{K}^2-\|f_{s_2}\|_{K}^2.
\end{align*}
Therefore $(f_s)_{s\in\N}$ is a
Cauchy sequence in $H(K)$, and 
\eqref{onein} as well as \eqref{oneone} follow.

We apply \eqref{onein} and \eqref{oneone} with $K^s$ instead of $K$
to derive $1 \in H(K^s)$ as well as $\|1\|_{K^s} = 1$. 
This yields $\psi_s^\X f = f \otimes 1 \in H(K)$ and
$\|\psi_s^\X f\|_K = \|f\|_{K_s}$
for $f \in H(K_s)$.

To establish the second part of the claim it suffices to show that
$K(\cdot,\by)$ belongs to the closure of
$\bigcup_{s\in \N}\psi_s^{\X}(H(K_s))$
for every $\by \in \X$.
Let $s \in \N$. Use the first part of the claim to
verify 
\[
\|K(\cdot,\by) - \psi^\X_s K_s (\cdot, (y_1,\dots,y_s))\|_K^2 =
K(\by,\by) - K_s((y_1,\dots,y_s),(y_1,\dots,y_s)).
\]
Since $\lim_{s \to \infty} K_s((y_1,\dots,y_s),(y_1,\dots,y_s))
= K(\by,\by)$, the statement follows.
\end{proof}

\subsection{Embeddings into $L^1$-spaces}

This section is based on \cite[Lem.~2.6 and Rem.~2.7]{HEF14}. 
A similar result, with a suboptimal constant, is presented in 
\cite[Lem.~7]{GMR12}.
The proof in the latter reference 
uses the Little Grothendieck Theorem. 
Here we give an elementary proof. The technique used in our proof is 
well-known and for instance used in the Malliavin calculus, 
and it is also a small part of a proof of the Little Grothendieck 
Theorem itself.

In this section we fix $s\in\N$.
For $j\in 1:s$
let $k_j$ be a reproducing kernel on $D_j\times D_j$
for some set $D_j\neq \emptyset$.
Furthermore assume that $\rho_j$ is a probability measure
(on a given $\sigma$-algebra) on $D_j$ such that
\begin{align*}
H(k_j) \subseteq L^1(D_j,\rho_j).
\end{align*}
Denote by $i_j$ the corresponding embedding. The closed graph theorem 
implies that $i_j$ is continuous.

Define the set
\begin{align*}
   D^{(s)}
=D_1\times\dots\times D_s
\end{align*}
as well as the product $\rho^{(s)}$ of the probability measures 
$\rho_1,\dots,\rho_s$.
We define the reproducing kernel $K_s$
on $D^{(s)}\times D^{(s)}$ by
\begin{align*}
K_s(\bx,\by) = \prod_{j= 1}^s
k_j(x_j,y_j), 
\end{align*}
where $\bx, \by\in D^{(s)}$.

\begin{lemma}\label{l100}
We have
\begin{align}\label{eqapp1}
H(K_s)
\subseteq L^1(D^{(s)},\rho^{(s)})
\end{align}
as well as
\begin{align}\label{eqapp2}
\|i\|
\leq
(\pi/2)^{(s-1)/2}
\prod_{j=1}^s\|i_j\|,
\end{align}
where $i$ denotes the embedding corresponding to \eqref{eqapp1}.
\end{lemma}

\begin{proof}
We proceed inductively, and here it suffices to consider
the case $s=2$.
Let $f_1,\dots,f_m\in H(k_1)$ be orthonormal,
$g_1,\dots,g_m\in H(k_2)$ be arbitrary,
and $X_1,\dots,X_m$ be independent
standard normally distributed random variables.
Set 
\[
c=\sqrt{{\pi}/{2}}\cdot \|i_1\|\|i_2\|
\]
and $h=\sum_{n=1}^m f_i\otimes g_i$.
We get
\begin{align*}
\int_{D_1 \times D_2} \left|h\right|\,\mathrm d \rho^{(2)}
& \leq \|i_1\|
\int_{D_2} \Bigl\|\sum_{n=1}^m f_n\,g_n(y)
\Bigr\|_{k_1}\rho_2(\mathrm dy)\\
& =\|i_1\|\int_{D_2}\Bigl(\sum_{n=1}^m 
g_n(y)^2\Bigr)^{1/2}\,\mathrm \rho_2(\mathrm dy)\\
&= \sqrt{\pi/2} \cdot \|i_1\|  
\int_{D_2} E\Bigl|\sum_{n=1}^m X_ng_n(y)\Bigr|\,\rho_2(\mathrm dy)\\
&= \sqrt{\pi/2} \cdot \|i_1\| \cdot
E\Bigl\|\sum_{n=1}^m X_ng_n\Bigr\|_{L^1(\rho_2)},
\end{align*}
and hence
\begin{align*}
\int_{D_1 \times D_2} \left|h\right|\,\mathrm d \rho^{(2)}
&\leq c \cdot E\Bigl\|\sum_{n=1}^m X_ng_n\Bigr\|_{k_2}
\leq c \biggl( E\Bigl\|\sum_{n=1}^m X_ng_n\Bigr\|_{k_2}^2
\biggr)^{1/2}\\
&= c \Bigl(\sum_{n=1}^m\|g_n\|_{k_2}^2\Bigr)^{1/2}
=c\left\|h\right\|_{K_s},
\end{align*}
which shows the claim.
\end{proof}

\begin{rem}
In the following we show that the constant $(\pi/2)^{(s-1)/2}$,
appearing in \eqref{eqapp2}, is optimal.
Let $m\in\N$ and $j\in 1:s$.
Let $\rho_j$ denote the standard normal distribution on $D_j =
\R^m$,
and let the reproducing kernel $k_j$ be such that
$H(k_j)$ is the dual space of $\R^m$
equipped with the Euclidean norm.
Furthermore, let $X$ denote a standard normally distributed
random variable.
Since $E(\lvert X\rvert)=\sqrt{{2}/{\pi}}$,
we get
\begin{align*}
\int_{D_j}  \left|f\right|\,\mathrm d\rho_j
=\sqrt{2/\pi}\cdot \|f\|_{k_j}
\end{align*}
for all $f\in H(k_j)$, and in particular,
\begin{align*}
\|i_j\|
=\sqrt{2/\pi}.
\end{align*}
Denote by $X_i\in H(k_j)$ the $i$th projection, i.e.,
\begin{align*}
X_i(\bx)
=x_i,
\end{align*}
for $\bx\in D_j$ and for $i=1,\dots,m$.
Note that $X_1,\dots, X_m\in H(k_j)$ are orthonormal.
Consider
\begin{align*}
h
=\sum_{n=1}^m X_n^{\otimes s},
\end{align*}
where $X_n^{\otimes s}$ denotes
the $s$-fold tensor product of $X_n$ with itself.
Note that
\begin{align*}
\left\|h/\sqrt{m}\right\|_{K_s}
=1,
\end{align*}
since $X_1^{\otimes s},\dots, X_m^{\otimes s}\in H(K_s)$ are orthonormal. 
The central limit theorem and a standard argument for convergence in 
distribution imply
\begin{align*}
\liminf_{m\to\infty}\int_{D^{(s)}} \left|h/\sqrt{m}\right|
\,\mathrm d \rho^{(s)}
\geq E(\lvert X\rvert)
=\sqrt{2/\pi}.
\end{align*}
This shows the claim, even if there is no dependence on $j$.
\end{rem}

\subsection*{Acknowledgment}
The authors thank Josef Dick 
and Greg Wasilkowski
for valuable discussions.

This paper was initiated at the Oberwolfach Workshop 1340 
``Uniform Distribution Theory and Applications'', and substantial 
progress was made during the Special Semester ``High-Dimensional 
Approximation'' at the Institute for Computational and Experimental 
Research in Mathematics (ICERM) at Brown University and at the 
Dagstuhl Seminar 15391 ``Algorithms and Complexity for 
Continuous Problems'' at Schloss Dagstuhl. The authors would like 
to thank the staff of the MFO Oberwolfach, the ICERM, and Leibniz-Zentrum 
f\"ur Informatik in Dagstuhl for providing a stimulating research 
environment and for their hospitality. 

\bibliographystyle{siam}
\bibliography{References}

\end{document}